\documentclass[reqno]{amsart}
\usepackage{amsmath,amsfonts,amssymb}
\usepackage{hyperref}
\usepackage{xcolor}

\usepackage[inner=1.0in,outer=1.0in,bottom=1.0in, top=1.0in]{geometry}
\newtheorem{theorem}{Theorem}[section]
\newtheorem{conjecture}[theorem]{Conjecture}
\newtheorem{proposition}[theorem]{Proposition}

\newtheorem{lemma}[theorem]{Lemma}
\newtheorem{remark}[theorem]{Remark}

\definecolor{blue}{rgb}{0,0,1}
\definecolor{red}{rgb}{1,0,0}
\definecolor{green}{rgb}{0,.6,.2}
\definecolor{purple}{rgb}{1,0,1}

\long\def\red#1\endred{\textcolor{red}{#1}}
\long\def\blue#1\endblue{\textcolor{blue}{#1}}
\long\def\purple#1\endpurple{\textcolor{purple}{ #1}}
\long\def\green#1\endgreen{\textcolor{green}{#1}}

\newcommand{\sm}{\left(\begin{smallmatrix}}
\newcommand{\esm}{\end{smallmatrix}\right)}
\newcommand{\bpm}{\begin{pmatrix}}
\newcommand{\ebpm}{\end{pmatrix}}

\numberwithin{equation}{section}

\title{Non-vanishing of symmetric cube $L$-functions}
\author{Jeff Hoffstein}
\address{Department of Mathematics, Brown University, Providence, RI 02912 USA}
\email{jeffrey{\textunderscore}hoffstein@brown.edu}
\author{Junehyuk Jung}
\address{Department of Mathematics, Brown University, Providence, RI 02912 USA}
\email{junehyuk{\textunderscore}jung@brown.edu}
\author{Min Lee}
\address{School of Mathematics, University of Bristol, Bristol, BS8 1UG, UK}
\email{min{\textunderscore}lee@bristol.ac.uk}

\thanks{
J. H. would like to thank S. Friedberg and D. Ginzburg for some stimulating conversations on the work of D. Ginzburg, D. Jiang and S. Rallis that is the basis of this paper.
We thank P. Sarnak for informing us of some of the  applications of non-vanishing of automorphic $L$-functions. 
J.J. was supported by NSF grant DMS-1900993, and by Sloan Research Fellowship. 
M.L. was supported by Royal Society University Research Fellowship ``Automorphic forms, $L$-functions and trace formulas''.}

\begin{document}
\begin{abstract}
We prove that there are infinitely many Maass--Hecke cuspforms over the field $\mathbb{Q}[\sqrt{-3}]$ such that the corresponding symmetric cube $L$-series does not vanish at the center of the critical strip.  This is done by using a result of  Ginzburg, Jiang and Rallis which shows that the symmetric cube non-vanishing happens if and only if a certain triple product integral involving the cusp form and the cubic theta function on $\mathbb{Q}[\sqrt{-3}]$ does not vanish.   We use spectral theory and the properties of the cubic theta function to show that the non-vanishing of this triple product occurs for infinitely many cusp forms.   We also formulate a conjecture about the meaning of the absolute value squared of the triple product which is reminiscent of Watson's identity.
\end{abstract}
\maketitle

\section{Introduction}
The non-vanishing of $L$-series at the center of the critical strip has long been a subject of great interest,  particularly when the degree of the Euler product $L$-series is even.
This is because, when normalized to have functional equations going from $s$ to $1-s$, under many circumstances when the degree of the Euler product is even the value at $s=\frac12$ is known, or conjectured, to have arithmetic significance.
When the degree is odd, for example, 1, it is the value at $s=1$, or the residue of a pole, that is known, or conjectured, to have arithmetic significance.
For example, for any ${\rm GL}(2)$ automorphic form there exists a half-integral weight Shimura correspondent if and only if there exists a quadratic twist of the corresponding $L$-series that does not vanish at the center.
A very important example of the significance of non-vanishing is in the case of an $L$-series corresponding to a modular form of weight 2,
where the non-vanishing at the central point has been shown to be equivalent to the finiteness of the group of rational points of the associated elliptic curve  \cite{MR463176,GZ}.

In the case of higher rank $L$-functions of even degree, such connections between non-vanishing at the center and the finiteness of certain groups are believed to be true, but the relations remain purely conjectural.
In particular, in the case of the symmetric cube $L$-series, Chao Li and Dorian Goldfeld  have informed us in a private correspondence that the Beilinson-Bloch conjecture predicts that the order of vanishing of the symmetric cube $L$-function should be equal to the rank of the Chow group of the  corresponding  symmetric cube motive.
In particular, if the modular form corresponds to an elliptic curve $E$, one looks at the group of homologically trivial algebraic cycles of dimension $1$ on the threefold $E \times E \times E$, where the symmetric group $S_3$ acts via the sign character.
Its rank is conjecturally the order of vanishing of the symmetric cube $L$-function at the central point.
There is some numerical evidence for this provided in a paper of Buhler, Schoen and Top \cite{BS}.

For other applications of the non-vanishing of $L$-functions on or near the critical line, including $L$-functions for symmetric powers of automorphic forms; see, for example, \cite{ps85}, \cite{lrs95} and \cite{lrs99}.

In \cite{GJR}, Ginzburg, Jiang and Rallis proved that the non-vanishing at the center of the critical strip of the symmetric cube $L$-series of any ${\rm GL}(2)$ automorphic form is equivalent to the non-vanishing of a certain triple product integral.
The main purpose of this paper is to use this equivalence to prove the following
\begin{theorem}\label{theorem:main}
Let $\Gamma_3 = {\rm SL}_2(\mathcal{O}_3)$ be the Bianchi group, where $\mathcal{O}_3$ is the ring of the integers of $\mathbb{Q}[\sqrt{-3}]$.
Let $\Gamma_3(3)$ be the level $3$ principal subgroup of $\Gamma_3$.
There are infinitely many  Maass--Hecke cuspforms $\phi_j$ on $\Gamma_3(3) \backslash \mathbb{H}^3$ such that
\[
L\left(\frac{1}{2}, \mathrm{sym}^3 ,\phi_j \right) \neq 0.
\]
\end{theorem}

In order to define the relevant triple product and further discuss our approach, it is necessary to first provide some background on what is known as cubic metaplectic forms defined over $\mathbb{Q}[\sqrt{-3}])$.
This is because a certain example of such a form, called a cubic theta function, is used to define the triple product studied by  Ginzburg, Jiang, and Rallis.

\subsection{Some background on Mass forms, cubic metaplectic Eisenstein series, and cubic theta functions defined over \texorpdfstring{$\mathbb{Q}[\sqrt{-3}]$}{the third cyclotomic field}.}
To understand what a cubic metaplectic form is, we first parameterize the upper half-space $\mathbb{H}^3$ using quaternions as follows:
\[
\mathbb{H}^3 = \{w=x_1+ix_2 + j y ~:~ y>0\}.
\]
Then the ${\rm SL}_2(\mathbb{C})$ action on $\mathbb{H}^3$ is given by
\[
\begin{pmatrix}a & b \\ c & d \end{pmatrix} w = (aw+b)(cw+d)^{-1},
\]
where $(cw+d)^{-1}$ is the multiplicative inverse of the quaternion $cw+d$.
We denote by $H(w) = y$ the $y$-component function on $\mathbb{H}^3$.
We identify the boundary of $\mathbb{H}^3$ with $\mathbb{C} \cup \{\infty\}$.
The Laplace--Beltrami operator on $\mathbb{H}^3$ is given by
\begin{equation}\label{e:Laplace-Beltrami}
\Delta = y^2 \left(\partial_{x_1}^2 + \partial_{x_2}^2 + \partial_{y}^2 \right) -y\partial_y,
\end{equation}
and the volume form is given by
\[
dV = \frac{dx_1dx_2dy}{y^3}.
\]
Denote by $\mathcal{O}_{d}$ the ring of integers of $\mathbb{Q}[\sqrt{-d}]$.
Then the Bianchi group $\Gamma_d = {\rm SL}_2(\mathcal{O}_{d})$ is a discrete subgroup of ${\rm SL}_2(\mathbb{C})$ such that the volume of the quotient space $\Gamma_d \backslash \mathbb{H}^3$ is finite.
For an ideal $I \subset \mathcal{O}_d$, the principal congruence subgroup of level $I$ in $\Gamma_d$ is given by
\[
\Gamma_d(I) = \left\{\gamma \in \Gamma_d ~:~ \gamma \equiv \begin{pmatrix}1 & 0 \\ 0 & 1 \end{pmatrix} \pmod{I}\right\}.
\]
In order to simplify the notations for the rest of the article, we use the following:
\begin{itemize}
\item $\lambda = \sqrt{-3}$.
\item $\omega = e^{\frac{2\pi i }{3}}$.
\item $\Lambda=\mathcal{O}_{3}= \mathbb{Z}[\omega]$ is the ring of integers of $K = \mathbb{Q} [\sqrt{-3}]$.
\item $\Gamma =  \Gamma_3(3)$ is the principal congruence subgroup of the level $(3)$ in $\Gamma_3 = {\rm SL}_2(\Lambda)$.
\item $e(z) = e^{2\pi i (z + \overline{z})}$, $z \in \mathbb{C}$.
\end{itemize}
We will mainly deal with  Maass forms and metaplectic Maass forms on $\Gamma \backslash \mathbb{H}^3$ in the subsequent sections.
The twelve equivalence classes of cusps of $\Gamma$ are given in \cite{pat1}, and are represented by
\begin{equation}\label{e:S_cusps}
\mathcal{S} = \{\infty,~ 0,~\pm 1,~\pm \omega,~\pm \omega^2,~\pm (1-\omega),~\pm (1-\omega)^{-1}\}.
\end{equation}
For $\mathfrak{a} \in \mathcal{S}$, we pick $\sigma_\mathfrak{a} \in \Gamma_3$ so that $\sigma_\mathfrak{a} \infty = \mathfrak{a}$.
To be specific, we let $\sigma_\mathfrak{a}$ be
\begin{equation}\label{e:sigma_a}
\begin{pmatrix} 1 & 0\\ 0& 1\end{pmatrix}, ~\begin{pmatrix} 0 & -1\\ 1& 0\end{pmatrix},~\begin{pmatrix} 1 & 0\\ \pm 1& 1\end{pmatrix},\begin{pmatrix} 1 &0 \\ \pm \omega^2 & 1\end{pmatrix},~\begin{pmatrix} 1 & 0\\ \pm \omega & 1\end{pmatrix},~\begin{pmatrix} \pm(1-\omega) & -1\\ 1& 0\end{pmatrix},~\begin{pmatrix} 1 & 0 \\ \pm (1-\omega)& 1\end{pmatrix},
\end{equation}
as done in \cite{pat1}.

Let $\Gamma_{\mathfrak{a}}$ be the stabilizer subgroup of $\Gamma$ corresponding to the cusp $\mathfrak{a}$.
Observe that $\Gamma$ is a normal subgroup of ${\rm SL}_2(\Lambda)$, and so we have
\[
\sigma_{\mathfrak{a}}^{-1} \Gamma_\mathfrak{a}\sigma_{\mathfrak{a}} = \Gamma_\infty.
\]

A Maass form $f$ on $\Gamma\backslash \mathbb{H}^3$ is a smooth function on $\mathbb{H}^3$ that satisfies the following conditions:
\begin{itemize}
\item $f(\gamma w) = f(w)$ for all $\gamma \in \Gamma$,
\item $-\Delta f = (1+4t^2) f$ for some $t\in \mathbb{C}$,  and
\item there exists $A>1$ such that $f(\sigma_\mathfrak{a} w) = O(H(w)^A)$ as $H(w)\to \infty$ for all $\mathfrak{a} \in \mathcal{S}$.
\end{itemize}
Note that $\Gamma_\infty$ is isomorphic to $3\Lambda$, and the dual lattice of $3\Lambda$ with respect to $e(\cdot)$ is $\lambda^{-3}\Lambda$.
So from the two conditions, we see that any Maass form $f$ on $\Gamma \backslash \mathbb{H}^3$ has a Fourier expansion at the cusp $\infty$ of the form
\begin{equation}\label{fourier}
f (w) = c_0 y^{1+2it} + c_{00} y^{1-2it} +\sum_{0\neq \mu \in \lambda^{-3} \Lambda} c_\mu y K_{2it} (4\pi |\mu| y) e(\mu x),
\end{equation}
where $x=x_1+ix_2$.

The Eisenstein series corresponding to the cusp $\mathfrak{a} \in \mathcal{S}$ is defined by
\begin{equation}\label{e:Ea_def}
E_{\mathfrak{a}}(w,s) = \sum_{\gamma \in \Gamma_{\mathfrak{a}} \backslash \Gamma} H(\sigma_{\mathfrak{a}}^{-1} \gamma w)^{2s},
\end{equation}
for $\mathrm{Re}(s)>1$.
Any Eisenstein series is a Maass form with the $-\Delta $-eigenvalue $4s(1-s)$.
In particular, the Eisenstein series corresponding to the cusp $\infty$ is defined by
\[
E(w,s) =E_\infty(w,s)= \sum_{\gamma \in \Gamma_{\infty} \backslash \Gamma} H( \gamma w)^{2s}.
\]
For the functions $F$ and $G$ on $\Gamma\backslash \mathbb{H}^3$ we denote by $\left\langle F, G\right\rangle$ the inner product
\begin{equation}\label{e:innerprod_def}
\left\langle F, G\right\rangle = \int_{\Gamma\backslash \mathbb{H}^3} F(w)\overline{G(w)} dV.
\end{equation}

\subsection{Metaplectic Maass forms}\label{ss:meta_maass}
Let $\kappa$ be a character on $\Gamma$ induced by the cubic residue symbol $(\cdot/\cdot)_3$ in $\mathbb{Q}[\sqrt{-3}]$, as introduced by Kubota in  \cite{kubotasym}.
In other words, for $\gamma= \sm a & b\\ c& d\esm \in \Gamma$,
\begin{equation}\label{kappa}
\kappa(\gamma)  = \begin{cases} \left(\frac{c}{d} \right)_3 & \text{ when } c\neq 0, \\ 1 & \text{ when } c=0. \end{cases}
\end{equation}

A Maass form $f$ on $\Gamma\backslash \mathbb{H}^3$ with respect to the character $\kappa$ (referred to as a metaplectic Maass form) is a smooth function on $\mathbb{H}^3$ that satisfies
\begin{itemize}
\item For any $\gamma \in \Gamma$,
\begin{equation}\label{e:auto_cond}
f(\gamma w) = \kappa(\gamma)f(w),
\end{equation}
\item $-\Delta f = (1+4t^2) f$ for some $t\in \mathbb{C}$,  and
\item there exists $A>1$ such that $f(\sigma_\mathfrak{a} w) = O(H(w)^A)$ as $H(w)\to \infty$ for all $\mathfrak{a} \in \mathcal{S}$.
\end{itemize}
In this paper, we consider two such functions.
The first is the metaplectic Eisenstein series:
\begin{equation}
E^{(3)}(w, s) = E^{(3)}_\infty(w, s)=\sum_{\gamma \in \Gamma_\infty \backslash \Gamma} \overline{\kappa(\gamma)} H(\gamma(w))^{2s}
\end{equation}
This term metaplectic essentially means that the Eisenstein series transforms with respect to the character $\kappa$ as follows:
$$
E^{(3)}(\gamma w, s) = \kappa(\gamma)E^{(3)}( w, s),
$$
where the notation in $\kappa(\gamma)$ is the same as in \eqref{kappa} above.
The other metaplectic form we will explore is the cubic theta series that we will define below.
\subsubsection{Some history of generalized metaplectic theta functions and Eisenstein series}
In the definitions above we used the cubic residue symbol to define $\kappa$.   If we had used the quadratic residue symbol instead we would have obtained the somewhat better known half-integral weight Eisenstein series defined over the ground field $\mathbb{Q}[\sqrt{-3}]$.

Rather than using 2 or 3, we could have defined $\kappa$ using an $n^{th}$ order residue symbol, for general $n$,
as long as we were working over a ground field containing the $n^{th}$ roots of unity.   If we had, we would have defined what is known as the $n^{th}$ order metaplectic Eisenstein series, also known as Eisenstein series on the $n$-fold metaplectic cover of ${\rm GL}_2(\mathbb{C})$.  These were first explored by Kubota  \cite{kubota}.
He observed that for $n\geq 2$, the Eisenstein series have a meromorphic continuation and he provided an explicit functional equation for them.
He computed their Fourier coefficients and discovered they are Dirichlet series with $n^{th}$ order Gauss sums as coefficients.   When $n = 2$ these series factors into Euler products and are essentially quadratic $L$-series, as was first observed by Maass \cite{Maass} working over $\mathbb{Q}$.
Siegel, in \cite{sie56} showed that taking the Mellin transform of the half-integral weight Eisenstein series created a Dirichlet series whose coefficients, at square free indices, were quadratic $L$-series.

However, when $n \geq 3$ the series in the Fourier coefficients do \emph{not} factor into an Euler product and are quite mysterious.   Nevertheless, the constant terms are expressible in terms of ratios of zeta functions of the ground field (any field containing the $n^{th}$ roots of 1), and have simple poles at the point $s = \frac{1}{2} + \frac{1}{2n}$.
Because of these poles in the constant term, the metaplectic Eisenstein itself has a pole at this point. Whichever $n$ we are working with, it is necessary to work over a base field that contains the $n^{th}$ roots of unity.  This is why we chose as a base field $\mathbb{Q}[\sqrt{-3}]$, as it is the simplest field that contains the cube roots of unity.  If we chose $n=2$, that is, if we took $\kappa$ to be induced by the quadratic residue symbol, the corresponding quadratic Eisenstein series could be defined over the rationals.  In this case it would be the usual half-integral weight Eisenstein series, which has a pole at $s=\frac34$ with residue equal to the usual Jacobi theta function over the base field $\mathbb{Q}$.

Kubota generalized the notion of a theta function by defining the $n^{th}$ order theta function to be the residue of the $n^{th}$ order metaplectic Eisenstein series at the point $s = \frac{1}{2} + \frac{1}{2n}$.
Kubota was not, however, able to determine the nature of the Fourier coefficients of these generalized theta functions.
In the case $n=3$, Patterson \cite{pat1} succeeded in computing the precise value of the Fourier coefficients of the residue of the cubic Eisenstein series, that is, the cubic theta function, up to the sign of the constant term (which he later determined in \cite{pat3}).
The foundation of this present paper is the evaluation of these coefficients.
Interestingly, to this date, the nature of the coefficients of $n^{th}$ order theta functions for general $n$ remains almost completely unknown.
There is a conjecture of Patterson in the case $n = 4$, and a conjecture of Chinta, Friedberg, and Hoffstein \cite{MR3060459} in the case $n = 6$, but there are not even conjectures for any other values of $n$.
See the introduction of \cite{Hoffstein-Reinier} for a brief history.

In the case we are considering, $E^{(3)}(w, s)$ has a simple pole at $s=\frac{2}{3}$, and the residue of $E^{(3)}(w, s)$ at this point is a cubic analog of the quadratic Jacobi theta function.
Denoting this by $\theta$, we have
\begin{equation}\label{theta}
\theta(w) = 2 \mathrm{Res}_{s=\frac{2}{3}} E^{(3)}(w, s).
\end{equation}
The Fourier expansion of the cubic theta series $\theta\in L^2 (\Gamma \backslash \mathbb{H}^3,\kappa)$ at $\infty$ is given by
\begin{equation}\label{exp:theta}
\theta (w) = \sigma y^{\frac{2}{3}}  +\sum_{0\neq \mu \in \lambda^{-3} \Lambda} \tau(\mu) y K_{\frac{1}{3}} (4\pi |\mu| y) e(\mu x).
\end{equation}
Here $\sigma= \frac{9\sqrt{3}}{2}$ \cite{pat3} and $\tau(\mu)$ is defined explicitly in \cite[Theorem 8.1]{pat1}.
Leaving out roots of $1$ and powers of $\lambda$  for simplicity,  for $\mu \in \Lambda$, if $\mu = m_0 m_1^3$, with $m_0, m_1 \equiv 1 \pmod 3$ and $m_0$ square free, then
\[
\tau(m_0m_1^3) = 27 \sqrt{Nm_1}\frac{\overline{g(m_0)}}{\sqrt{Nm_0}}.
\]
Here $Nm = m\overline{m}= |m|^2$ for $m\in \mathbb{Q}[\sqrt{-3}]$.
The coefficient vanishes if $m_0$ is cube-free but not square-free.
Here $g(m_0)$ is the cubic Gauss sum
\[
g(m_0) = \sum_{\alpha \bmod {m_0}}\left(\frac{\alpha}{m_0} \right)_3e\left(\frac{\alpha}{m_0} \right).
\]
The absolute values of the $\tau(\mu)$ are all we will need, and these are given as follows.
For $a \in \{0,1,2\}$, $m_0, m_1 \equiv 1 \pmod 3$ and $m_0$ square free,
\begin{equation}\label{|tau|def}
|\tau(\mu)| = \begin{cases} 3^{2+\frac{n}{2}} (Nm_1)^{\frac12} & \text{ when } \mu = \pm \omega^a\lambda^{3n-4}m_0m_1^3, n\ge1 \\
3^{\frac{n+5}{2}}(Nm_1)^{\frac12} & \text{ when } \mu = \pm \lambda^{3n-3}m_0m_1^3, n\ge0, \\
0& \text{otherwise}. \end{cases}
\end{equation}

\subsection{The approach, and a conjecture}
To begin to discuss our attack on the problem of proving the non-vanishing of the symmetric cube $L$-series at the center of the critical strip, we first recall the main result of Ginzburg, Jiang and Rallis \cite{GJR}:
\begin{theorem}[Ginzburg, Jiang and Rallis]\label{theorem:mainlem}
For a Maass--Hecke cuspform  $\phi\in L^2(\Gamma \backslash \mathbb{H}^3)$,
\[
L\left(\frac{1}{2}, \mathrm{sym}^3 ,\phi \right) \neq 0
\]
if and only if
\[
\langle  \phi ,|\theta|^2 \rangle \neq 0.
\]
Here $\theta\in L^2(\Gamma \backslash \mathbb{H}^3,\kappa)$ is the cubic theta series defined above.
\end{theorem}

\begin{remark}
Note that the inner product is well-defined because $|\theta|^2$ is invariant under $\Gamma$ by the automorphic condition \eqref{e:auto_cond}.
\end{remark}

We will prove Theorem~\ref{theorem:main} by first arguing that a weighted average of $\langle \phi_j,|\theta|^2\rangle$
 with Laplace eigenvalue $1+4t_j^2$  over $|t_j|<2r$ must grow with $r$, and then using Theorem~\ref{theorem:mainlem}

The equivalence of the non-vanishing of $L\left(\frac{1}{2}, \mathrm{sym}^3 ,\phi \right)$
and $\langle |\theta|^2 , \phi\rangle \neq 0$ suggests that there may be an identity relating the two.
We formalize this in the following
\begin{conjecture}\label{mainconj}
Let $\phi$ be a Maass cusp form with ground field $K$ containing the cube roots of unity.
Then
\[
\left|\langle \phi, |\theta|^2 \rangle\right|^2 = c_{\phi}\frac{L^*\left(\frac12,\phi,\mathrm{sym}^3 \right)}{L^*\left(1,\phi,\mathrm{sym}^2 \right)}.
\]
Here the $L$-series are, respectively the completed symmetric cube and symmetric square $L$-series of $\phi$, and the constant $c_\phi$ is non zero and depends on local data of $\phi$ at the prime $3$.
\end{conjecture}

\subsection{A heuristic supporting Conjecture \ref{mainconj}}
Suppose we replace $\phi$, with spectral parameter $1+2it$,
by the non-metaplectic Eisenstein series $E(w,s)$, which  has as parameter $s$, with Laplace eigenvalue $2s(2-2s)$.

For $\mathrm{Re}(s) >1$ the inner product $\langle E(w,s), |\theta|^2 \rangle$ unfolds to
\[
\int_{\Gamma_\infty  \backslash \mathbb{H}^3}  y^{2s} |\theta(w)|^2 dV.
\]
Some caution must be used here. The inner product must be regularized. (We do this formally in \S\ref{ss:proof_lemma_cont}.)

The Eisenstein series must be approximated by the truncated function
\[
E^T(w, s) = \sum_{\gamma \in \Gamma_{\infty} \backslash \Gamma} I_T(H( \gamma w))^{2s},
\]
where $I_T$ is the characteristic function of the interval $[T^{-1},T]$, and $T \rightarrow \infty$.

When this is done, after unfolding $E^T(w,s)$, the inner product becomes
\[
\iiint_{\Gamma_\infty  \backslash \mathbb{H}^3}
\left|\sigma y^{\frac{2}{3}}  +\sum_{0\neq \mu \in \lambda^{-3} \Lambda} \tau(\mu) y K_{\frac{1}{3}} (4\pi |\mu| y) e(\mu x)\right|^2I_T(y)^{2s}\frac{dx_1dx_2dy}{y^3}.
\]
Letting $T\rightarrow \infty$ carefully, the term containing the square of the constant term vanishes due to an analytic continuation argument and all terms disappear except the non-zero diagonal terms.
This leaves us with
\begin{multline*}
\left\langle E(w,s), |\theta|^2 \right\rangle
=
\mathrm{Vol}(3\Lambda)\int_0^\infty  y^{2s} \sum_{0\neq \mu \in \lambda^{-3} \Lambda} |\tau(\mu)|^2  K_{\frac{1}{3}} (4\pi |\mu| y)^2 \frac{dy}{y}
\\ =\frac{9\sqrt{3}}{2} \sum_{0\neq \mu \in \lambda^{-3} \Lambda} |\tau(\mu)|^2\int_0^\infty  y^{2s} K_{\frac{1}{3}} (4\pi |\mu| y)^2 \frac{dy}{y}
=\frac{9\sqrt{3}}{2(4\pi)^{2s}} \sum_{0\neq \mu \in \lambda^{-3} \Lambda}\frac{ |\tau(\mu)|^2}{N\mu^s}
\int_0^\infty  y^{2s} K_{\frac{1}{3}} (y)^2 \frac{dy}{y},
\end{multline*}
after interchanging the order of integration and summation, changing variables and substituting $\frac{9\sqrt{3}}{2}$ for the value of $\mathrm{Vol}(3\Lambda)$.

By \eqref{eq2}, given below,
\[
\int_0^\infty  y^{2s} K_{\frac{1}{3}} (y)^2 \frac{dy}{y}= \frac{2^{-3+2s}\Gamma\left(s+\frac13\right)\Gamma(s)^2\Gamma\left(s-\frac13\right)}{\Gamma(2s)}.
\]
The triplication formula for the Gamma function states that
\[
\Gamma(z)\Gamma\left(z+ \frac13\right)\Gamma\left(z+\frac23\right) = 2\pi3^{\frac12-3z}\Gamma(3z).
\]
Applying this to the above, with $z=s-\frac13$, gives us
\begin{equation}\label{e:K1/3_Mellin}
\int_0^\infty  y^{2s} K_{\frac{1}{3}} (y)^2 \frac{dy}{y}
= \frac{2^{-2+2s}\pi 3^{\frac32-3s}\Gamma(3s-1)\Gamma(s)}{\Gamma(2s)}.
\end{equation}
Also, referring to \eqref{|tau|def},
\[
\sum_{0\neq \mu \in \lambda^{-3} \Lambda}\frac{ |\tau(\mu)|^2}{(N\mu)^s}
= \sum_{\substack{\mu = \pm \omega^a\lambda^\alpha, \\ a=0,1,2, \alpha \ge -3 }}
\frac{ |\tau(\mu)|^2}{3^{\alpha s}}
\sum_{\substack{m = m_0m_1^3, \\ m_0, m_1 \equiv 1 \bmod 3, \\ m_0 \text{ cube free}}}
\frac{Nm_1}{(Nm_0)^s (Nm_1)^{3s}}.
\]
The $3$-part sums to
\[
2 \cdot3^{5+3s}\left(1+3^{1-2s}\right)\left(1-3^{1-3s}\right)^{-1},
\]
while the part relatively prime to $3$ sums to
\[
\frac{\zeta^{(3)}_K(3s-1)\zeta^{(3)}_K(s)}{\zeta^{(3)}_K(2s)},
\]
where $\zeta^{(3)}_K$ is the zeta function of the field $K = \mathbb{Q} [\sqrt{-3}]$ \eqref{e:zetaK_def}, with the Euler factor at the prime $\lambda$ removed.
Assembling the above, since
\[
\zeta^{(3)}_K(s)= \zeta_K(s)\left(1-3^{-s}\right),
\]
we have
\begin{equation}\label{e:L_theta^2}
\sum_{0\neq \mu \in \lambda^{-3} \Lambda}\frac{ |\tau(\mu)|^2}{(N\mu)^s}
= 2 \cdot3^{5+3s}\frac{(1+3^{1-2s})(1-3^{-s})}{(1-3^{-2s})}
\frac{\zeta_K(3s-1)\zeta_K(s)}{\zeta_K(2s)}.
\end{equation}
We finally have
\begin{equation}\label{left}
\left\langle E(\cdot,s), |\theta|^2 \right\rangle
= 3^9 2^{-2-2s}\pi^{1-2s}\frac{(1+3^{1-2s})(1-3^{-s})}{(1-3^{-2s})}
\frac{\zeta_K(3s-1)\Gamma(3s-1) \zeta_K(s)\Gamma(s)}{\zeta_K(2s)\Gamma(2s)}.
\end{equation}
Recalling the completed zeta function of the number field $K$ is
\begin{equation}\label{e:zetaK*}
\zeta_K^*(s) =  \left(\frac{3}{4\pi^2}\right)^{\frac{s}{2}}  \Gamma(s) \zeta_K(s),
\end{equation}
we rewrite \eqref{left} as
\begin{equation}\label{left2}
\left\langle E(\cdot,s), |\theta|^2 \right\rangle
 =3^{-s+\frac{19}{2}} 2^{-3}\frac{(1+3^{1-2s})(1-3^{-s})}{(1-3^{-2s})}
\frac{\zeta_K^*(3s-1) \zeta_K^*(s)}{\zeta_K^*(2s)}.
\end{equation}
Setting $s = \frac{1}{2} + it$ and multiplying \eqref{left2} by its conjugate, we obtain
\begin{equation}\label{left3}
 \left|\left\langle E(\cdot, 1/2 + it), |\theta|^2 \right\rangle\right|^2
= \frac{3^{18}\left|1+3^{-2it}\right|^2 \left|1-3^{-\frac{1}{2} -it}\right|^2}
{2^6\left|1-3^{-1-2it}\right|^2}
\frac{\zeta^{*}_K(\frac12 + 3it)\zeta^{*}_K(\frac12 - 3it)\zeta^{*}_K(\frac{1}{2}+it)\zeta^{*}_K(\frac12 -it)}{\zeta^{*}_K(1+2it)\zeta^{*}_K(1-2it)}.
\end{equation}

The $L$-series, in a new variable $u,$ attached to $E(w,s)$, which we denote for convenience as $E(s)$, is
\[
L(u,E(s))= \zeta_K\left(u+s-\frac12\right)\zeta_K\left(u-s+\frac12\right).
\]
To make $E(s)$ resemble a Maass form with spectral parameter $t$ we set $s = \frac12+it$ and have
\[
L\left(u,E\left(\frac12 +it\right)\right)= \zeta_K\left(u+it\right)\zeta_K\left(u-it\right).
\]
We can now take the symmetric square $L$-series, getting
\[
L\left(u,E\left(\frac12 +it\right), \mathrm{sym}^2 \right)= \zeta_K\left(u+2it\right)\zeta_K(u)\zeta_K\left(u-2it\right),
\]
and finally the symmetric cube:
\[
L\left(u,E\left(\frac12 +it\right), \mathrm{sym}^3 \right)= \zeta_K\left(u+3it\right)\zeta_K\left(u+it\right)\zeta_K\left(u-it\right)\zeta_K\left(u-3it\right).
\]
The $L$-series attached to $E\left(\frac12 +it\right)$ at the center of the critical strip is
$\zeta_K\left(\frac12 +it\right)\zeta_K\left(\frac12 -it\right)$.
Similarly the symmetric square $L$-series has a pole, with residue $\zeta_K\left(\frac12 +2it\right)\zeta_K\left(\frac12 -2it\right)$, that is, $L^*(2u, E(1/2+it),  \mathrm{sym}^2)$ has a pole at $u=1/2$.
The symmetric cube $L$-series is
$\zeta_K\left(\frac12 +3it\right)\zeta_K\left(\frac12 +it\right)\zeta_K\left(\frac12 -it\right)\zeta_K\left(\frac12 -3it\right)$.
Thus \eqref{left3} can be rewritten as
\begin{equation}\label{left4}
 \left|\langle E(\cdot, 1/2+it), |\theta|^2 \rangle\right|^2
= c_{E\left( \frac12 +it \right)}\frac{
L^*\left(\frac12,E\left(\frac12 +it\right), \mathrm{sym}^3 \right)}{2{\rm Res}_{u=\frac12}L^*\left(2u,E\left(\frac12 +it\right), \mathrm{sym}^2 \right)}.
\end{equation}
with $L^*$ denoting the completed $L$-series and
\[
 c_{E(\frac{1}{2}+it)} =  \frac{3^{18}\left|1+3^{-2it}\right|^2 \left|1-3^{-\frac{1}{2} -it}\right|^2}{2^6\left|1-3^{-1-2it}\right|^2}.
\]
The pole at $u = \frac12$ comes about because the Eisenstein series is not a cusp form.
It seems reasonable to believe that the appropriate substitute for the residue of the symmetric square in the case of an Eisenstein series would be the symmetric square $L$-series itself in the case of a cusp form that is not a lift from ${\rm GL}(1)$,  which leads us to Conjecture~\ref{mainconj}.

\subsection{A road map of the approach}
After establishing some basic facts about the Fourier coefficients of the theta function $\theta(w)$ and the metaplectic and non-metaplectic Eisenstein series, and the spectral theory of $L^2 (\Gamma \backslash \mathbb{H}^3)$,
we define a Poincar\'e series $P_{\mu} (w,s)$ in \eqref{PoincDef}, and consider its inner product with $|\theta(w)|^2$, namely  $|\langle P_\mu(\cdot, s) ,|\theta|^2  \rangle|$.
As explained in Lemma~\ref{lemma:lem1} this picks off the $\mu$ coefficient of $|\theta(w)|^2$, along with some gamma factors.
We then derive the spectral expansion of  $P_{\mu} (w,s)$ in \eqref{poin} and compute the inner product $|\langle P_\mu(\cdot, s) ,|\theta|^2  \rangle|$ in another way, using this expansion.
We show that this breaks up into a continuous piece plus a discrete piece.
Setting $s = \alpha +ir$,  for sufficiently large and fixed $\alpha >\frac{2}{3}$,  we show in Lemmas~\ref{lemma:cont} and \ref{lemma:spec}  that the continuous piece contribution is $\mathcal{O}_\mu \left(e^{-\pi |r|} (1+|r|)^{2\alpha-2}\right)$.
In this same lemma, it is shown that the remainder of the contribution to the spectral expansion is a linear combination over $j$ of the inner products   $\overline{\rho_j(\mu)} \langle \phi_j,|\theta|^2 \rangle$.

In Lemma~\ref{lemma:direct} we compute the inner product differently, by multiplying $\theta(w)$ by its conjugate and using
$P_{\mu} (w,s)$ to pick off the $\mu$ coefficient, for any choice of $\mu$, such as $\mu = 1$.
From this we obtain  a collection of shifted sums, and verify that there is a main term and an error term, and
\[
 |\langle P_\mu(\cdot, s) ,|\theta|^2  \rangle| \sim_{\alpha,\mu} (1+|r|)^{2\alpha - \frac{4}{3}} e^{-\pi |r|}.
\]
Comparing this with the continuous contribution, which is  $\mathcal{O}\left(e^{-\pi |r|} (1+|r|)^{2\alpha-2}\right)$,
as $2 >\frac43$ this means that the discrete contribution contributes the difference,
which implies  $\overline{\rho_j(\mu)} \langle \phi_j,|\theta|^2 \rangle \ne 0$ infinitely often.

\subsection{Notation and miscellaneous lemmas}
As always, $A\ll_\tau B$ means that $|A|<CB$ for some constant $C=C(\tau)$ depending only on $\tau$. We write $A\sim_\tau B$ when $|A| \ll_\tau |B|$ and $|B| \ll_\tau |A|$.
We will frequently use Stirling's approximation, in the following form.
\begin{lemma}[Stirling]\label{Stirling}
Fix $0<a<b$. For $\alpha \in [a,b]$, and $r \in \mathbb{R}$, with large $|r|$, we have
\[
|\Gamma(\alpha+ir)| \sim e^{-\frac{\pi |r|}{2} } (1+|r|)^{\alpha - \frac{1}{2}}.
\]
\end{lemma}
Finally, we also use the following two integration formulas involving $K$-Bessel functions:
\begin{equation}\label{eq1}
\int_0^\infty e^{-y} K_{2it} (y) (2y)^{s} \frac{dy}{y^2} = 2\sqrt{\pi} \frac{\Gamma (s-1+2it) \Gamma (s-1-2it)}{\Gamma \left(s - \frac{1}{2}\right)}.
\end{equation}
from \cite[p.700 6.621.3.]{GR7}, and
\begin{equation}\label{eq2}
\int_0^\infty K_\mu (x) K_\nu (x) x^{s} \frac{dx}{x}
= 2^{-3+s} \frac{\Gamma \left(\frac{s+\mu+\nu}{2}\right)\Gamma \left(\frac{s-\mu+\nu}{2}\right)\Gamma \left(\frac{s+\mu-\nu}{2}\right)\Gamma \left(\frac{s-\mu-\nu}{2}\right)}{\Gamma(s)}.
\end{equation}
from \cite[p.684 6.576.4.]{GR7}.

\section{Spectral summation}\label{sec3}
The main purpose of this section is to represent the inner product $\langle  P_\mu(\cdot, s),|\theta|^2 \rangle$ between an incomplete Poincar\'e series (defined in \S\ref{def}) and $|\theta|^2$ as a spectral sum.
Note that $|\theta|^2$ is not an $L^2$ integrable function, so one can not directly apply a Parseval-like theorem.
We obtain such a spectral summation formula by first spectrally expanding $P_{\mu}(w,s)$, and then by taking the inner product with $|\theta|^2$.

\subsection{Basic spectral theory}
We begin by reviewing the spectral theory of the Laplace-Beltrami operator $\Delta$ \eqref{e:Laplace-Beltrami} on $L^2 (\Gamma \backslash \mathbb{H}^3)$.
We first describe the Fourier expansion of the Eisenstein series as follows.
Recall that $\mathcal{S}$ is the set of cusps \eqref{e:S_cusps},
and $E_{\mathfrak{a}}(w, s)$ is the Eisenstein series at $\mathfrak{a}\in \mathcal{S}$, defined in \eqref{e:Ea_def}.
\begin{proposition}\label{lem:eisfourier}
Let $\zeta_K(s)$ be the Dedekind zeta function associated to the imaginary quadratic field $\mathbb{Q}[\sqrt{-3}]$
\begin{equation}\label{e:zetaK_def}
\zeta_K(s) = \sum_{0\neq c \in \Lambda} \frac{1}{(Nc)^{s}},
\end{equation}
and let $\zeta_K^*(s) = \left(\frac{3}{4\pi^2}\right)^{\frac{s}{2}}\Gamma(s) \zeta_K(s)$ be the completed zeta function.
Then the Fourier expansion of $E_\mathfrak{a}$ at $\infty$ is given by
\begin{equation}\label{exp:eis}
E_\mathfrak{a}(w,s) = \delta_{\mathfrak{a},\infty} y^{2s} + c_{\mathfrak{a}}(0,s) y^{2-2s} +\sum_{0\neq \mu \in \lambda^{-3} \Lambda} c_\mathfrak{a}(\mu,s) y K_{2s-1} (4\pi |\mu| y) e(\mu x),
\end{equation}
where
\begin{equation}\label{e:c_a0s}
c_{\mathfrak{a}}(0,s) = \frac{\zeta_K^*(2s-1)}{\zeta_K^*(2s)} \tilde{c}_{\mathfrak{a}}(0,s)
\end{equation}
and
\begin{equation}\label{e:c_amus}
c_\mathfrak{a}(\mu,s) = \frac{1}{\zeta_K^*(2s)} \tilde{c}_\mathfrak{a}(\mu,s),
\end{equation}
with $\tilde{c}_\mathfrak{a}(\mu,s)$ being a Dirichlet polynomial in $s$.
\end{proposition}

\begin{proof}
Following the standard computation \cite{kubota}, we see that the $\mu$-th Fourier coefficient of $E_\mathfrak{a}$ at $\infty$ is given by
\[
\frac{1}{\mathrm{Vol}(3\Lambda)}\iint_{\mathbb{C}/3\Lambda }\sum_{\gamma \in \Gamma_{\mathfrak{a}} \backslash \Gamma}  H(\sigma_{\mathfrak{a}}^{-1} \gamma w)^{2s} e(-\mu x ) dx_1dx_2
\]
which is equal to
\[
y^{2s} \delta_{\mathfrak{a},\infty} + \psi_{\mathfrak{a}\infty}(0,s)  \frac{\pi\Gamma(2s-1)}{\mathrm{Vol}(3\Lambda)\Gamma(2s)}y^{2-2s}
\]
for $\mu=0$, and
\[
\frac{(2\pi)^{2s} |\mu|^{2s-1}}{\mathrm{Vol}(3\Lambda)\Gamma(2s)} y K_{2s-1}(4\pi |\mu| y) \psi_{\mathfrak{a}\infty}(\mu,s),
\]
for $\mu\neq 0$, where
\[
\psi_{\mathfrak{a}\infty}(\mu,s) = \sum_{\gamma \in \Gamma_\infty \backslash \sigma_\mathfrak{a}^{-1} \Gamma /\Gamma_\infty } \frac{e\left(\frac{\mu d}{c}\right)}{|c|^{4s}},
\]
for $\gamma = \begin{pmatrix}*&*\\ c&d\end{pmatrix}$.
If we write $\sigma_\mathfrak{a}^{-1} = \begin{pmatrix}* &*\\ \alpha &\beta \end{pmatrix}$, then $\Gamma_\infty \backslash \sigma_\mathfrak{a}^{-1} \Gamma /\Gamma_\infty$ is parameterized by
$(c,d) \in \Lambda^2$ such that $\mathrm{gcd}(c,d) = 1$, $c \equiv \alpha \pmod{3}$ and $d \pmod{3c}$ with $d\equiv \beta\pmod{3}$.
From this, we may arrange the summation so that
\[
\psi_{\mathfrak{a}\infty}(\mu, s) = \sum_{\substack{0\neq c \in \Lambda\\ c \equiv \alpha(3) }}
 \frac{1}{Nc^{2s}}\sum_{\substack{d(3c)\\d\equiv \beta(3),~ \mathrm{gcd}(c,d) = 1}}e\left(\frac{\mu d}{c}\right).
\]
One may then verify the claim via directly evaluating Ramanujan's sums
\[
\sum_{\substack{d(3c)\\d\equiv \beta(3),~ \mathrm{gcd}(c,d) = 1}}e\left(\frac{\mu d}{c}\right),
\]
and then expressing $\psi_{\mathfrak{a}\infty}(\mu,s)$ as a ratio of a Dirichlet polynomial and $\zeta_K^*(2s)$.
\end{proof}

Because $\zeta_K^*(s)$ is meromorphic on $\mathbb{C}$ with simple poles only at $s=0$ and $1$,
and because $\zeta_K^*(s)$ does not vanish if $\mathrm{Re}(s)\geq 1$, we infer from Proposition~\ref{lem:eisfourier} the meromorphic continuation of $E_{\mathfrak{a}}$.

\begin{proposition}\label{noresidual}
For $\mathfrak{a} \in \mathcal{S}$, $E_{\mathfrak{a}}(w, s)$ admits a meromorphic continuation to $\mathrm{Re}(s)\geq \frac{1}{2}$ with a simple pole only at $s=1$.
\end{proposition}
This implies that, aside from a constant function, there is no $L^2$-integrable eigenfunction of the Laplace--Beltrami operator which is a residue of an Eisenstein series.

Let $\Theta \subset L^2 (\Gamma \backslash \mathbb{H}^3)$ be the subspace spanned by incomplete theta series
\[
\theta_{\mathfrak{a},\psi} (w) =  \sum_{\gamma \in \Gamma_{\mathfrak{a}} \backslash \Gamma} \psi(H(\sigma_{\mathfrak{a}}^{-1} \gamma w))
\]
with $\psi \in C_0^\infty (0,\infty)$.

Let $\hat{\Theta}\subset \Theta$ be the subspace spanned by the residues of Eisenstein series at $s \in (1/2,1]$, and let $\Theta_0$ be the orthogonal complement of $\hat{\Theta}$ in $\Theta$.
Let $L_{{\rm cusp}}^2(\Gamma \backslash \mathbb{H}^3)$, the space of cusp forms, i.e., the subspace of square integrable functions $f \in L^2(\Gamma \backslash \mathbb{H}^3)$ such that
\[
\iint_{\sigma_\mathfrak{a}^{-1}\Gamma_\mathfrak{a}\sigma_\mathfrak{a}\backslash \mathbb{C}} f(\sigma_\mathfrak{a} (x_1+ix_2+jy)) dx_1dx_2 =0
\]
for almost all $y$, for all $\mathfrak{a} \in \mathcal{S}$.

\begin{proposition}
We have the following direct sum of subspaces
\[
L^2(\Gamma \backslash \mathbb{H}^3) = \hat{\Theta} \oplus \Theta_0 \oplus L_{{\rm cusp}}^2(\Gamma \backslash \mathbb{H}^3).
\]
Here $\hat{\Theta}$ is one dimensional and consists of constant functions.
The spectrum of $\Delta$ on $\Theta_0$ is purely continuous, and the spectrum of $\Delta$ on $ L_{{\rm cusp}}^2(\Gamma \backslash \mathbb{H}^3)$ is discrete.
\end{proposition}

It is well-known that the Eisenstein series $E_\mathfrak{a} \left(w, \frac{1}{2}+it\right)$ for $\mathfrak{a} \in \mathcal{S}$ span $\Theta_0$.
Let $\{\phi_j\}_{j\geq 1}$ be an orthonormal basis of $ L_{{\rm cusp}}^2(\Gamma \backslash \mathbb{H}^3)$ that consists of Maass--Hecke eigenforms.
Such a basis exists because $\Delta$  and the Hecke operators $\{T_\mu\}$ form a commuting family of self-adjoint operators.
Let $\phi_0 = (\mathrm{Vol}(\Gamma \backslash \mathbb{H}^3))^{-1/2}$.
We summarize the spectral expansion of a square-integrable function on $\Gamma \backslash \mathbb{H}^3$ in the following proposition.

\begin{proposition}\label{prop:exp}
For $F$ in $L^2 (\Gamma \backslash \mathbb{H}^3)$, we have
\begin{equation}\label{exp}
F(w)= \sum_{j\geq 0} \langle F, \phi_j \rangle \phi_j(w) + \frac{1}{4\pi} \sum_{\mathfrak{a}\in \mathcal{S}}\int_{-\infty}^\infty \langle F, E_{\mathfrak{a}} (\cdot, 1/2+it)\rangle E_{\mathfrak{a}} (w, 1/2+it) dt,
\end{equation}
in the sense of $L^2$, i.e.,
\[
\langle F,G\rangle= \sum_{j\geq 0} \langle F, \phi_j \rangle \langle \phi_j, G\rangle  + \frac{1}{4\pi} \sum_{\mathfrak{a}\in \mathcal{S}}\left\langle\int_{-\infty}^\infty \langle F, E_{\mathfrak{a}} (\cdot, 1/2+it)\rangle E_{\mathfrak{a}} (\cdot, 1/2+it) dt, G\right\rangle
\]
for all $G \in L^2(\Gamma\backslash \mathbb{H}^3)$.
If we further assume that $F$ is smooth, then  \eqref{exp} is true pointwise, and the right-hand side converges absolutely.
\end{proposition}

\subsection{Incomplete Poincar\'e series}\label{def}

For $\mu \in \lambda^{-3} \Lambda$, and $s \in \mathbb{C}$, we define the incomplete Poincar\'e series by
\begin{equation}\label{PoincDef}
P_{\mu} (w,s) = \sum_{\gamma  \in \Gamma_\infty \backslash \Gamma} F_s(\gamma w)
\end{equation}
where $F_s(w) = (8\pi |\mu| y)^{2s} e^{-4\pi|\mu| y} e(\mu x)$, where $x=x_1+ix_2$, and $w=x+yj$.

\begin{lemma}\label{lemma:lem1}
Assume that $\mu \neq 0$, and let $f$  be a Maass form on $\Gamma \backslash \mathbb{H}^3$ with the Fourier expansion \eqref{fourier}.
Assume that $it \in (-1/2,1/2)\cup i\mathbb{R}$, and that $\mathbb{Re}(s)>3/2$.
Then we have
\[
\langle P_{\mu}(\cdot,s), f \rangle = 36 \sqrt{3} \pi^{\frac{3}{2}} |\mu|\overline{c_\mu}
\frac{\Gamma (2s-1+2it) \Gamma (2s-1-2it)}{\Gamma \left(2s - \frac{1}{2}\right)}.
\]
\end{lemma}

\begin{proof}
By unfolding the integral, we first have
\begin{align*}
\langle P_{\mu}(\cdot,s),f \rangle &= \iiint_{\Gamma \backslash \mathbb{H}^3} P_{\mu}(w,s)\overline{f(w)} \frac{dx_1dx_2dy}{y^3}\\
&=\iiint_{\Gamma_\infty  \backslash \mathbb{H}^3} (8\pi |\mu| y)^{2s}e^{-4\pi|\mu| y} e(\mu x)\overline{f(w)} \frac{dx_1dx_2dy}{y^3}.
\end{align*}
Since we assumed that $\mu \neq 0$, the integral over $x_1$ and $x_2$ simplifies the expression to
\begin{align*}
&=\mathrm{Vol}(3\Lambda)\int_0^\infty (8\pi |\mu| y)^{2s} e^{-4\pi|\mu| y} \overline{c_\mu} y K_{2it} (4\pi |\mu| y)   \frac{dy}{y^3}\\
&=18 \sqrt{3} \pi |\mu|\overline{c_\mu} \int_0^\infty   e^{-y} K_{2it} ( y)  ( 2y)^{2s} \frac{dy}{y^2},
\end{align*}
where we used $\mathrm{Vol}(3\Lambda) = 9\sqrt{3}/2$ in the last equality. Now the statement follows from  \eqref{eq1}.
\end{proof}

\subsection{The inner product formula}
We assume that the Fourier expansion of $\phi_j$ ($j\geq 1$) is given by
\begin{equation}\label{exp:maass}
\phi_j = \sum_{0\neq \mu \in \lambda^{-3} \Lambda} \rho_j(\mu) y K_{2it_j} (4\pi |\mu| y) e(\mu x).
\end{equation}

We apply Lemma \ref{lemma:lem1} to Proposition \ref{prop:exp} with $F=P_{\mu}(w,s)$ so that
\begin{multline}\label{poin}
P_{\mu}(w,s)= \sum_{j} \langle P_{\mu}(\cdot,s), \phi_j \rangle \phi_j(w) + \frac{1}{4\pi} \sum_{\mathfrak{a}\in \mathcal{S}}\int_{-\infty}^\infty \langle P_{\mu}(\cdot,s), E_{\mathfrak{a}} (\cdot, 1/2+it)\rangle E_{\mathfrak{a}} (\cdot, 1/2+it) dt\\
=\frac{36 \sqrt{3} \pi^{3/2} |\mu|}{\Gamma \left(2s - \frac{1}{2}\right)}\sum_{j}   \Gamma (2s-1+2it_j) \Gamma (2s-1-2it_j)  \overline{\rho_j(\mu)} \phi_j(w)\\
 + \frac{9 \sqrt{3} \pi^{1/2} |\mu|}{\Gamma \left(2s - \frac{1}{2}\right)} \sum_{\mathfrak{a}\in \mathcal{S}}\int_{-\infty}^\infty\overline{c_\mathfrak{a}\left(\mu,1/2+it\right)}  \Gamma (2s-1+2it) \Gamma (2s-1-2it)  E_{\mathfrak{a}} (w, 1/2+it) dt.
\end{multline}
Because $|\theta|^2$ does not belong to $L^2(\Gamma \backslash \mathbb{H}^3)$, in order to express $\langle P_{\mu}(\cdot,s), |\theta|^2\rangle $ as the summation of the inner product between each summand in \eqref{poin} and $|\theta|^2$,
we need to understand how $\rho_j(\mu)$ and $\langle  \phi_j , |\theta|^2 \rangle$ behave as $j \to \infty$.
We begin with an estimate of the Fourier coefficients.

\begin{lemma}\label{firstcoe}
As $t_j \to \infty$, we have
\[
|\rho_j(\mu)| \ll t_j^{\frac{1}{3}} e^{\pi t_j} |\mu|.
\]
\end{lemma}

\begin{proof}
by the standard upper bound for the supnorm of an eigenfunction on finite volume symmetric spaces \cite{MR0232893},
\[
\mathrm{Vol}(3\Lambda)\sum_{0\neq \mu \in \lambda^{-3} \Lambda} |\rho_j(\mu)|^2 y^2 K_{2it_j} (4\pi |\mu| y)^2=\iint_{\mathbb{C}/3\Lambda} |\phi_j|^2 dx_1dx_2 \leq  \sup|\phi_j|^2 \ll |t_j|^2
\]
Assume without loss of generality that $t_j>100$, and,  for a given $\mu\in \lambda^{-3}\Lambda$,
choose $y$ so that $4\pi |\mu| y = 2t_j$. Then we have
\[
|\rho_j(\mu)|^2 \frac{t_j^2}{4\pi^2|\mu|^2} K_{2it_j} (2t_j)^2 \ll t_j^2,
\]
and so using the asymptotic of the K-Bessel function in the transition range \cite{MR698780},
\[
|\rho_j(\mu)| \ll t_j^{\frac{1}{3}} e^{\pi t_j} |\mu|.\qedhere
\]
\end{proof}

It will be convenient to know how the Fourier expansion of $\theta(w)$ with respect to the cusp $\mathfrak{a} \in \mathcal{S}$ looks like.
\begin{lemma}\label{lem:theta_fourier}
For $\mathfrak{a} \in \mathcal{S}$, let the Fourier expansion of $\theta$ with respect to $\mathfrak{a}$ given by
\begin{equation}\label{thetafour}
\theta(\sigma_{\mathfrak{a}}w) = c_\mathfrak{a} y^{\frac{2}{3}} + \sum_{0\neq \mu \in \lambda^{-3}\Lambda} \tau_\mathfrak{a}(\mu) y K_{\frac{1}{3}}(4\pi |\mu| y) e(\mu x).
\end{equation}
Then we have $|\tau_{\mathfrak{a}}(\mu)| = |\tau(\mu)|$, for all $\mu \neq 0$.
\end{lemma}

\begin{proof}
Because of the relation \eqref{theta}
\[
\theta(w) = 2 \mathrm{Res}_{s=\frac{2}{3}} E^{(3)}(w, s),
\]
we have
\[
\tau_\mathfrak{a}(\mu) y K_{\frac{1}{3}}(4\pi |\mu| y) =2 \mathrm{Res}_{s=\frac{2}{3}} \frac{1}{\mathrm{Vol}(3\Lambda)}\iint_{\mathbb{C}/3\Lambda} E^{(3)}(\sigma_{\mathfrak{a}}w,s) e(-\mu x) dx.
\]
As done in \cite{kubota}, we see that
\[
\iint_{\mathbb{C}/3\Lambda} E^{(3)}(\sigma_{\mathfrak{a}}w,s) e(-\mu x) dx = \frac{(2\pi)^{2s} |\mu|^{2s-1}}{\mathrm{Vol}(3\Lambda)\Gamma(2s)} y K_{2s-1}(4\pi |\mu| y) \psi_{\infty\mathfrak{a}}^{(3)}(\mu,s),
\]
where
\[
\psi_{\infty\mathfrak{a}}^{(3)}(\mu,s) = \sum_{\gamma \in \Gamma_\infty \backslash  \Gamma \sigma_{\mathfrak{a}}/\Gamma_\infty }\overline{\kappa(\gamma\sigma_{\mathfrak{a}}^{-1})} \frac{e\left(\frac{\mu d}{c}\right)}{|c|^{4s}},
\]
for $\gamma = \begin{pmatrix}*&*\\ c&d\end{pmatrix}$. If we write $\sigma_\mathfrak{a} = \begin{pmatrix}* &*\\ \alpha &\beta \end{pmatrix}$, then $\Gamma_\infty \backslash  \Gamma\sigma_\mathfrak{a} /\Gamma_\infty$ is parameterized by $(c,d) \in \Lambda^2$ such that $\mathrm{gcd}(c,d) = 1$, $c \equiv \alpha \pmod{3}$ and $d \pmod{3c}$ with $d\equiv \beta\pmod{3}$. From this, we may arrange the summation so that
\begin{equation}\label{psi}
\psi_{\mathfrak{a}\infty}^{(3)}(\mu,s) = \sum_{\substack{0\neq c \in \Lambda\\ c \equiv \alpha(3) }}
 \frac{1}{|c|^{4s}}\sum_{\substack{d(3c)\\d\equiv \beta(3),~ \mathrm{gcd}(c,d) = 1}} \overline{\kappa(\gamma\sigma_{\mathfrak{a}}^{-1})}e\left(\frac{\mu d}{c}\right).
\end{equation}
One can check for each $\mathfrak{a}$ that $\alpha$, $\beta$, and $\kappa$ are given by the following table.
\begin{center}
\begin{tabular}{c|c|c|c}
$\mathfrak{a}$ & $\alpha$ & $\beta$ & $\overline{\kappa(\gamma\sigma_{\mathfrak{a}}^{-1})}$\\ \hline
$\infty$ & $0$&$1$&$(c/d)_3$\\
$0$& $1$& $0$&$(d/c)_3$\\
$\pm 1$&$\pm 1$ &$1$&$(c/d)_3$\\
$\pm \omega$& $\pm \omega^2$& $1$&$(c/d)_3$\\
$\pm \omega^2$& $\pm \omega$&$1$&$(c/d)_3$\\
$\pm(1-\omega)$&$1 $&$0$&$(d/c)_3$\\
$\pm(1-\omega)^{-1}$&$\pm (1-\omega)$ &$1$&$(c/d)_3$
\end{tabular}
\end{center}
Comparing \eqref{psi} with (5.4) of \cite{pat1}, we see that for $\mathfrak{a} \neq \infty,~\pm(1-\omega)^{-1} $, we have
\[
|\psi_{\mathfrak{a}\infty}^{(3)}(\mu,s)| = \frac{1}{\mathrm{Vol}(3\Lambda)} |\psi(s,\mu,0)|
\]
where $\psi(s,\mu,l)$ is defined by (5.18) of \cite{pat1}. This proves $|\tau_\mathfrak{a}(\mu)|=|\tau(\mu)|$. Now when $\mathfrak{a} = \pm(1-\omega)^{-1}$, we can express $\psi_{\mathfrak{a}\infty}^{(3)}(\mu,s)$ as a linear combination of $\psi(s, \varepsilon \lambda^b \mu, 0)$ as in (5.24) of \cite{pat1}, where $\varepsilon$ is a unit and $b \geq 1$. Then the equation follows by computing the residue of the summation $\psi(s, \varepsilon \lambda^b \mu, 0)$ at $s=2/3$.
\end{proof}

Finally, we bound the contribution coming from the continuous spectrum as follows.

\begin{lemma}\label{lemma:cont}
Let $s=\alpha+ir$ with  sufficiently large $\alpha>\frac{2}{3}$  being fixed.
Then for any $\mathfrak{a} \in \mathcal{S}$, as $r \to \infty$, we have
\[
\left\langle\int_{-\infty}^\infty\overline{c_\mathfrak{a}\left(\mu,\frac{1}{2}+it\right)}  \Gamma (2s-1+2it) \Gamma (2s-1-2it)  E_{\mathfrak{a}} (\cdot , 1/2+it) dt, |\theta|^2\right\rangle \ll_{\mu, \epsilon} e^{-2\pi |r|}(1+|r|)^{4\alpha -3+\epsilon},
\]
for any $\epsilon>0$.
\end{lemma}
Because the proof is quite lengthy, we present the proof of this lemma in \S\ref{ss:proof_lemma_cont}.

We collect these estimates to derive the spectral summation formula for the inner product $\langle P_{\mu}(\cdot,s),|\theta|^2 \rangle$.
\begin{lemma}\label{lemma:spec}
Fix a sufficiently large $\alpha>0$.
We have
\[
\langle P_{\mu}(\cdot,s),|\theta|^2 \rangle=\frac{36 \sqrt{3} \pi^{3/2} |\mu|}{\Gamma \left(2s - \frac{1}{2}\right)}\sum_{j}   \Gamma (2s-1+2it_j) \Gamma (2s-1-2it_j)  \overline{\rho_j(\mu)} \langle \phi_j,|\theta|^2 \rangle + O_{\mu, \epsilon} (e^{-\pi |r|} (1+|r|)^{2\alpha-2+\epsilon}),
\]
 for any $\epsilon>0$,
and the summation converges absolutely.
\end{lemma}
\begin{proof}
We first have
\begin{multline}\label{thesum}
\langle P_{\mu}(\cdot,s),|\theta|^2\rangle=\frac{36 \sqrt{3} \pi^{3/2} |\mu|}{\Gamma \left(2s - \frac{1}{2}\right)}\left\langle\sum_{j}   \Gamma (2s-1+2it_j) \Gamma (2s-1-2it_j)  \overline{\rho_j(\mu)} \phi_j,|\theta|^2\right\rangle\\
 + \frac{9 \sqrt{3} \pi^{1/2} |\mu|}{\Gamma \left(2s - \frac{1}{2}\right)} \sum_{\mathfrak{a}\in \mathcal{S}}\left\langle\int_{-\infty}^\infty\overline{c_\mathfrak{a}\left(\mu, 1/2+it\right)}  \Gamma (2s-1+2it) \Gamma (2s-1-2it)  E_{\mathfrak{a}} (\cdot , 1/2+it) dt,|\theta|^2\right\rangle.
\end{multline}
Because, by Lemma~\ref{Stirling},
\[
\frac{1}{\Gamma\left(2s-\frac{1}{2}\right)}\ll e^{\pi |r|} (|r|+1)^{1-2\alpha},
\]
we bound the second term by  $O_{\mu, \epsilon} (e^{-\pi |r|} (1+|r|)^{2\alpha-2+\epsilon})$ using Lemma~\ref{lemma:cont}.

Now we have
\[
\Gamma (2s-1+2it_j) \Gamma (2s-1-2it_j)  \ll e^{-\pi(|r+t_j|+|r-t_j|) }  ((|r+t_j|+1)(|r-t_j|+1))^{2\alpha - \frac{3}{2}},
\]
and by the supnorm estimate \cite{MR0232893}, we have
\begin{equation}\label{triv}
\langle\phi_j, |\theta|^2 \rangle \leq \|\theta\|_{L^2}^2 \sup|\phi_j| \ll |t_j|.
\end{equation}
Therefore we infer from Lemma \ref{firstcoe} that
\[
\left\langle \Gamma (2s-1+2it_j) \Gamma (2s-1-2it_j)  \overline{\rho_j(\mu)} \phi_j,|\theta|^2\right\rangle
\]
decays exponentially in $t_j$, hence the summation in \eqref{thesum} is absolutely convergent, from which we may interchange the order of the inner product and the summation.
\end{proof}

\subsection{Proof of Lemma~\ref{lemma:cont}}\label{ss:proof_lemma_cont}

We fix a fundamental domain $\mathcal{F}\cong\Gamma\backslash \mathbb{H}^3$.
Let $\pm \triangle$ be the triangle in $\mathbb{C}$ with vertices $0$, $\pm(1-\omega)^{-1}$ and $\pm(1-\omega^{2})^{-1}$.
Following \cite[(2.2)]{pat1}, let
\begin{equation}\label{e:scrF0}
\mathcal{F}_0 = \left\{x+yj\in \mathbb{H}^3: |x|^2+y^2>1, x\in (+\triangle)\cup(-\triangle)\right\}
\end{equation}
and
\[
\Gamma_{3, \infty} = \left\{\bpm \epsilon & \nu\\ 0 & \epsilon^{-1}\ebpm : \epsilon \text{ unit }, \nu\in \Lambda\right\} \subset \Gamma_3 = {\rm SL}_2(\Lambda).
\]
Then, as in \cite[p.130]{pat1},
\begin{equation}\label{e:scrF}
\mathcal{F} = \bigcup_{\mathfrak{a}\in \mathcal{S}} \bigcup_{m\in \Gamma_{3, \infty}/\Gamma_\infty} \sigma_{\mathfrak{a}} m \mathcal{F}_0
\end{equation}
is a fundamental domain for $\Gamma$.
For $\eta>0$, let
\[
D_\eta = \left\{(x, y)\in \mathbb{H}^3: x\in U, y>\eta\right\},
\]
where $(+\triangle)\cup (-\triangle)\subset U\cong\mathbb{C}/3\Lambda$.
There exists $0<\eta<1$ such that $\mathcal{F}_0\subset D_\eta$, and we let
\begin{equation}\label{e:Siegeldomain}
\mathcal{D} = \bigcup_{\mathfrak{a}\in \mathcal{S}} \bigcup_{m\in \Gamma_{3, \infty} /\Gamma_\infty} \sigma_{\mathfrak{a}} m D_\eta.
\end{equation}
Then $\mathcal{F} \subset \mathcal{D}$.
Moreover, for any $w\in \sigma_{\mathfrak{a}} m D_\eta$, $H(\sigma_{\mathfrak{a}}^{-1} w)>\eta$.

For $w\in \mathbb{H}^3$, let
\begin{equation}
\mathcal{E}_{\mathfrak{a}}(w; \mu, s)
= \frac{1}{2\pi} \int_{-\infty}^\infty\overline{c_\mathfrak{a}(\mu,1/2+it)}  \Gamma (2s-1+2it) \Gamma (2s-1-2it)  E_{\mathfrak{a}} \left(w, 1/2+it\right) dt
\end{equation}

\begin{lemma}\label{lem:innerprod_scrEatheta^2}
For any cusp $\mathfrak{a}\in \mathcal{S}$,  the inner product
\[
\left\langle\mathcal{E}_{\mathfrak{a}}(\cdot; \mu, s), |\theta|^2\right\rangle
= \iiint_{\Gamma\backslash \mathbb{H}^3} \mathcal{E}_{\mathfrak{a}}(w; \mu, s) |\theta(w)|^2 dV
\]
converges absolutely  for $\Re(s)> \frac{2}{3}$.
\end{lemma}

\begin{proof}
For $\mu\neq 0$, by \eqref{e:c_amus},
\[
c_\mathfrak{a}(\mu, 1/2+it) = \frac{\tilde{c}_\mathfrak{a}(\mu,1/2+it)}{\zeta_K^*(1+2it)}
\]
where $\tilde{c}_{\mathfrak{a}}(\mu, 1/2+it)$ is a Dirichlet polynomial in $it$.
Let $\overline{\tilde{c}_{\mathfrak{a}}}(\mu, 1/2+it)$ be a Dirichlet polynomial in $it$ such that
\[
\overline{\tilde{c}_{\mathfrak{a}}}(\mu, 1/2-it) = \overline{\tilde{c}_{\mathfrak{a}}(\mu, 1/2+it)}.
\]
Then
\[
\overline{c_\mathfrak{a}(\mu, 1/2+it)} = \frac{\overline{\tilde{c}_\mathfrak{a}}(\mu,1/2-it)}{\zeta_K^*(1-2it)}.
\]
Here $\zeta_K^*(s) = \left(\frac{3}{4\pi^2}\right)^{\frac{s}{2}}\Gamma(s) \zeta_K(s)$ is the completed zeta function.
Note that $\zeta_K^*(s)$ has simple poles only at $s\in \{0, 1\}$.

Recalling the Fourier expansion of $E_{\mathfrak{a}}(w; 1/2+it)$ in \eqref{exp:eis},
\begin{multline}\label{e:scrEa_decomp}
\mathcal{E}_{\mathfrak{a}}(w; \mu, s)
= \delta_{\mathfrak{a},\infty} \frac{1}{2\pi i} \int_{(0)} \frac{\overline{\tilde{c}_{\infty}}(\mu, 1/2-z)}{\zeta_K^*(1-2z)}
\Gamma\left(2s-1+2z\right) \Gamma\left(2s-1-2z\right) y^{1+2z} dz
\\ + \frac{1}{2\pi i} \int_{(0)} \frac{\overline{\tilde{c}_{\mathfrak{a}}}(\mu, 1/2-z)c_{\mathfrak{a}}(0,1/2+z) }{\zeta_K^*(1-2z)}
\Gamma\left(2s-1+2z\right) \Gamma\left(2s-1-2z\right) y^{1-2z} dz
\\ + \frac{1}{2\pi} \int_{-\infty}^\infty \frac{\overline{\tilde{c}_{\mathfrak{a}}(\mu, 1/2+it)}}{\zeta_K^*(1-2it)}
\Gamma\left(2s-1+2it\right) \Gamma\left(2s-1-2it\right)
\sum_{0\neq \nu \in \lambda^{-3} \Lambda} c_\mathfrak{a}(\nu, 1/2+it) y K_{2it} (4\pi |\nu| y) e(\nu x) dt.
\end{multline}
By the functional equation, we have $\zeta_K^*(1-2z) = \zeta_K^*(2z)$, and so we have
\[
\frac{\overline{\tilde{c}_{\mathfrak{a}}}(\mu, 1/2-z)c_{\mathfrak{a}}(0,1/2+z) }{\zeta_K^*(1-2z)}
= \frac{\overline{\tilde{c}_{\mathfrak{a}}}(\mu, 1/2-z)\tilde{c}_{\mathfrak{a}}(0,1/2+z)}{\zeta_K^*(1+2z)},
\]
where $\tilde{c}_{\mathfrak{a}}(0, 1/2+z)$ is a Dirichlet polynomial for \eqref{e:c_a0s}.
Let $\alpha=\Re(s)>0$.
Move the $z$-line of integration for the first integral in \eqref{e:scrEa_decomp} to $\mathrm{Re}(z)=-\alpha+\frac{1}{2}+\epsilon$,
and move the $z$-line of integration for the second integral to $\mathrm{Re}(z) = \alpha-\frac{1}{2}-\epsilon$ for sufficiently small $\epsilon>0$.
Note that we do not pass over any poles.
The series in the third integral converges absolutely and the size is $O(e^{-2\pi y})$.
We get
\[
\mathcal{E}_{\mathfrak{a}}(w; \mu, s)
\ll_{\mathfrak{a}, \mu, \alpha, \epsilon} y^{2-2\alpha+\epsilon},
\]
for $y>1$.
Since $|\theta(w)|^2=O(y^{\frac{4}{3}})$ as $y\to \infty$, for $\alpha>\frac{2}{3}$, $y>1$, we have
\[
\left|\theta(w)\right|^2 \mathcal{E}_{\mathfrak{a}}(w; \mu, s)  \ll_{\mathfrak{a}, \mu, \alpha, \epsilon} y^{\frac{10}{3}-2\alpha+\epsilon}.
\]

We now consider the asymptotic behaviour of $|\theta(w)|^2 \mathcal{E}_{\mathfrak{a}}(w; \mu)$ as $w$ approachies a cusp $\mathfrak{b}\in \mathcal{S}$.
We consider $|\theta(\sigma_{\mathfrak{b}} w)|^2 \mathcal{E}_{\mathfrak{a}}(\sigma_{\mathfrak{b}} w; \mu, s)$ as $w\to \infty$.
By the definition of the Eisenstein series, for $\mathrm{Re}(s)>1$, we have
\[
E_{\mathfrak{a}}(\sigma_{\mathfrak{b}}w, s)
= \sum_{\gamma\in \Gamma_{\mathfrak{a}}\backslash \Gamma} H(\sigma_{\mathfrak{a}}^{-1}\gamma \sigma_{\mathfrak{b}} w)^{2s}
= \sum_{\gamma\in \Gamma_{\mathfrak{a}}\backslash \Gamma} H(\sigma_{\mathfrak{a}}^{-1}\sigma_{\mathfrak{b}} \cdot \sigma_{\mathfrak{b}}^{-1} \gamma \sigma_{\mathfrak{b}} w)^{2s}.
\]
For $\Gamma=\Gamma_3(3)$, we have $\sigma_{\mathfrak{b}}^{-1} \Gamma \sigma_{\mathfrak{b}} = \Gamma$ for any $\sigma_{\mathfrak{b}}$ in \eqref{e:sigma_a}.
There exist $\mathfrak{c}\in \mathcal{S}$ and $\gamma_{\mathfrak{c}} \in \Gamma$ such that
$\mathfrak{c} = \gamma_{\mathfrak{c}} \sigma_{\mathfrak{b}}^{-1} \sigma_{\mathfrak{a}}\infty = \gamma_{\mathfrak{c}} \sigma_{\mathfrak{b}}^{-1}\mathfrak{a}$.
Then $\sigma_{\mathfrak{c}} = \gamma_{\mathfrak{c}} \sigma_{\mathfrak{b}}^{-1}\sigma_{\mathfrak{a}}$.
For $\gamma_{\mathfrak{c}}\sigma_{\mathfrak{b}}^{-1}\Gamma_{\mathfrak{a}} \sigma_{\mathfrak{b}}\gamma_{\mathfrak{c}}^{-1}=\Gamma_{\mathfrak{c}}$, we have
\begin{equation}\label{e:Eab=Ec}
E_{\mathfrak{a}}(\sigma_{\mathfrak{b}}w, s)
= \sum_{\gamma\in \Gamma_{\mathfrak{c}}\backslash \Gamma} H(\sigma_{\mathfrak{c}}^{-1}\gamma \gamma_{\mathfrak{c}} w)^{2s}
= E_{\mathfrak{c}}(w, s).
\end{equation}
For any $\mathfrak{b}\in \mathcal{S}$, by \eqref{thetafour}, $|\theta(\sigma_{\mathfrak{b}} w)|^2=O(y^{\frac{4}{3}})$ as $y\to\infty$, and following the previous argument, we have
\begin{equation}\label{e:theta^2scrEa_upper}
\left|\theta(\sigma_{\mathfrak{b}} w)\right|^2 \mathcal{E}_{\mathfrak{a}}(\sigma_{\mathfrak{b}} w; \mu, s)
= \left|\theta(\sigma_{\mathfrak{b}} w)\right|^2 \mathcal{E}_{\mathfrak{c}}(w; \mu, s)
\ll_{\mathfrak{c}, \mu, \alpha, \epsilon} y^{\frac{10}{3}-2\alpha+\epsilon}.
\end{equation}

By applying \eqref{e:theta^2scrEa_upper}, and using the description of the fundamental domain $\mathcal{F}$ \eqref{e:scrF}
and the Siegel domain \eqref{e:Siegeldomain},
\begin{align}
\left|\left\langle\mathcal{E}_{\mathfrak{a}}(\cdot; \mu, s), |\theta|^2\right\rangle\right|
& \leq \int_{\mathcal{F}} \left|\mathcal{E}_{\mathfrak{a}}(w; \mu, s)\right| |\theta(w)|^2 dV  \nonumber
\\ & \leq \sum_{\mathfrak{b}\in \mathcal{S}} \sum_{m\in \Gamma_{3, \infty} /\Gamma_\infty}
\iiint_{\sigma_{\mathfrak{b}} m D_\eta} \left|\mathcal{E}_{\mathfrak{a}}(w; \mu, s)\right| |\theta(w)|^2 \frac{dx_1dx_2dy}{y^3} \nonumber
\\ & \ll \sum_{\mathfrak{b}\in \mathcal{S}} \sum_{m\in \Gamma_{3, \infty} /\Gamma_\infty}
\int_\eta^\infty \left|\mathcal{E}_{\mathfrak{a}}(\sigma_{\mathfrak{b}} w; \mu, s)\right| |\theta(\sigma_{\mathfrak{b}} w)|^2 \frac{dy}{y^3} \nonumber
\\ & \ll_{\mu, \alpha, \epsilon} \sum_{\mathfrak{b}\in \mathcal{S}} \sum_{m} \int_\eta^\infty y^{\frac{1}{3}-2\alpha+\epsilon} dy
\ll_{\mu, \alpha, \epsilon} \eta^{\frac{4}{3}-2\alpha+\epsilon}, \label{e:inner_scrEtheta^2_conv}
\end{align}
for $\alpha>\frac{2}{3}$.
So the inner product $\left\langle\mathcal{E}_{\mathfrak{a}}(\cdot; \mu, s), |\theta|^2\right\rangle$ converges absolutely as claimed.
\end{proof}

Our goal is to express the inner product $\left\langle\mathcal{E}_{\mathfrak{a}}(\cdot; \mu, s), |\theta|^2\right\rangle$ as an absolutely convergent integral involving zeta functions, and then estimate in terms of $s$.
We follow Arthur's method for treating the truncated Eisenstein series.
See, for example, \cite{Gol}.

For $T>1$, let $1_T(y)$ be the characteristic function such that
\[
1_T(y) = \begin{cases}
1 & \text{ when } y>T, \\
0 & \text{ otherwise.}
\end{cases}
\]
For $\mathfrak{a}\in \mathcal{S}$, define
\[
\Lambda_{\mathfrak{a}}^T |\theta|^2 (w) = \sum_{\gamma\in \Gamma_{\mathfrak{a}} \backslash \Gamma} 1_T(H(\sigma_{\mathfrak{a}}^{-1} \gamma w)) H(\sigma_{\mathfrak{a}}^{-1}\gamma w)^{\frac{4}{3}}.
\]
Fix a compacta $C\subset \Gamma \backslash \mathbb{H}^3$.
For $T>1$, we see that there are only finitely many $\gamma \in \Gamma_{\mathfrak{a}} \backslash\Gamma$, $\sigma_{\mathfrak{a}}^{-1}\gamma= \sm a & b\\ c& d\esm$, such that
$H(\sigma_{\mathfrak{a}}^{-1} \gamma w) =  \frac{y}{|cx+d|^2+|c|^2y^2} > T$, since there are only finitely many $c, d\in \Lambda$ satisfying
\[
|c|^2 \left(y^2+\left|x+\frac{d}{c}\right|^2\right) < \frac{y}{T}.
\]
Thus $\Lambda_{\mathfrak{a}}^T|\theta|^2(w)$ is a finite sum for $w\in C$, and the number of the terms depends only on $C$ and $T$.
Define
\[
\Lambda^T|\theta|^2(w) = \sum_{\mathfrak{b}\in \mathcal{S}} |c_{\mathfrak{b}}|^2 \Lambda_{\mathfrak{b}}^T|\theta|^2(w)
\]
and consider $(|\theta|^2 - \Lambda^T|\theta|^2)(w)$.
Following the arguments in the proof of Lemma~\ref{lem:innerprod_scrEatheta^2}, we can show that the inner product $\left\langle\mathcal{E}_{\mathfrak{a}}(\cdot; \mu), |\theta|^2-\Lambda^T|\theta|^2\right\rangle$ converges absolutely.
Similarly, following \eqref{e:inner_scrEtheta^2_conv},
\[
\left|\left\langle\mathcal{E}_{\mathfrak{a}}(\cdot; \mu, s), \Lambda^T|\theta|^2\right\rangle\right|
\ll T^{\frac{4}{3}-2\alpha+\epsilon}.
\]
For $\alpha>\frac{2}{3}$, we get
\[
\lim_{T\to \infty} \left\langle\mathcal{E}_{\mathfrak{a}}(\cdot; \mu, s), \Lambda^T|\theta|^2\right\rangle=0.
\]
So we have
\begin{equation}\label{e:limT_inner_scrELambdaT}
\lim_{T\to \infty} \bigg(\left\langle\mathcal{E}_{\mathfrak{a}}(\cdot; \mu, s), |\theta|^2-\Lambda^T|\theta|^2\right\rangle\bigg)
= \left\langle\mathcal{E}_{\mathfrak{a}}(\cdot; \mu, s), |\theta|^2\right\rangle - \lim_{T\to \infty} \left\langle\mathcal{E}_{\mathfrak{a}}(\cdot; \mu, s), \Lambda^T|\theta|^2\right\rangle
= \left\langle\mathcal{E}_{\mathfrak{a}}(\cdot; \mu, s), |\theta|^2\right\rangle.
\end{equation}

\begin{lemma}\label{lem:inner_Eatheta^2}
The inner product $\left\langle E_{\mathfrak{a}}(\cdot; 1/2+it), |\theta|^2-\Lambda^T|\theta|^2\right\rangle$ converges absolutely for any sufficiently large $T>1$.
\end{lemma}

\begin{proof}
For any cusp $\mathfrak{c}\in \mathcal{S}$, on $w\in  \bigcup_{m\in \Gamma_{3, \infty}/\Gamma_\infty} \sigma_{\mathfrak{c}}m\mathcal{F}_0$, for $\mathcal{F}_0$ as given in \eqref{e:scrF0},
\begin{multline*}
(|\theta|^2 - \Lambda^T|\theta|^2)(w)
= |\theta(w)|^2 - |c_{\mathfrak{c}}|^2 1_T(H(\sigma_{\mathfrak{c}}^{-1}w) H(\sigma_{\mathfrak{c}}^{-1} w)^{\frac{4}{3}}
- \sum_{\mathfrak{c}\neq \mathfrak{b}\in \mathcal{S}} |c_{\mathfrak{b}}|^2 1_T(H(\sigma_{\mathfrak{b}}^{-1} w)) (H(\sigma_{\mathfrak{b}}^{-1} w)^{\frac{4}{3}}
\\ = \big(1-1_{T}(H(\sigma_{\mathfrak{c}}^{-1}w))\big) H(\sigma_{\mathfrak{c}}^{-1}w)^{\frac{4}{3}}
- \sum_{\mathfrak{c}\neq \mathfrak{b}\in \mathcal{S}} |c_{\mathfrak{b}}|^2 1_T(H(\sigma_{\mathfrak{b}}^{-1} w)) H(\sigma_{\mathfrak{b}}^{-1} w)^{\frac{4}{3}}
+ O(e^{-2\pi H(\sigma_{\mathfrak{c}}^{-1}w)}),
\end{multline*}
as $H(\sigma_{\mathfrak{c}}^{-1}w)\to \infty$.
For sufficiently large $X>1$, there exists $0<\delta <T$ such that for any $w$ with $H(\sigma_{\mathfrak{c}}^{-1} w)>X$, $H(\sigma_{\mathfrak{b}}^{-1}w)<\delta$.
Then $1_T(H(\sigma_{\mathfrak{b}}^{-1}w))=0$.
So as $H(\sigma_{\mathfrak{c}}^{-1}w)\to \infty$, we have
\[
(|\theta|^2 - \Lambda^T|\theta|^2)(w)
= O(e^{-2\pi H(\sigma_{\mathfrak{c}}^{-1}w)}).
\]
Therefore we get
\begin{multline}
\left|\left\langle E_{\mathfrak{a}}(\cdot; 1/2+it), |\theta|^2-\Lambda^T|\theta|^2\right\rangle\right|
\leq \int_{\mathcal{F}} |E_{\mathfrak{a}}(w; 1/2+it)|
\bigg||\theta|^2(w) - \sum_{\mathfrak{b}\in \mathcal{S}} |c_{\mathfrak{b}}|^2 1_T(H(\sigma_{\mathfrak{b}}^{-1} w)) \bigg| dV
\\ \leq \sum_{\mathfrak{c}\in \mathcal{S}} \sum_m \int_{\sigma_{\mathfrak{c}} m D_c} |E_{\mathfrak{a}}(w; 1/2+it)|
\bigg||\theta|^2(w) - \sum_{\mathfrak{b}\in \mathcal{S}} |c_{\mathfrak{b}}|^2 1_T(H(\sigma_{\mathfrak{b}}^{-1} w)) \bigg|  dV
\\ \ll \sum_{\mathfrak{c}\in \mathcal{S}} \sum_m
\int_c^T |E_{\mathfrak{a}}(\sigma_{\mathfrak{c}}w; 1/2+it)| y^{\frac{4}{3}} \frac{dy}{y^3}
+ \int_T^\infty |E_{\mathfrak{a}}(\sigma_{\mathfrak{c}} w; 1/2+it)| e^{-2\pi y} \frac{dy}{y^3} < \infty.
\end{multline}
Thus the inner product $\left\langle E_{\mathfrak{a}}(\cdot; 1/2+it), |\theta|^2-\Lambda^T|\theta|^2\right\rangle$ converges absolutely for any $T>c$.
\end{proof}

By Lemma~\ref{lem:inner_Eatheta^2} and the argument above, for both inner products converge absolutely,
we interchange the order of the integral and the inner product:
\begin{multline}\label{e:inner_scrE_withT}
\left\langle\mathcal{E}_{\mathfrak{a}}(\cdot; \mu, s), |\theta|^2-\Lambda^T|\theta|^2\right\rangle
\\ = \left\langle\frac{1}{2\pi} \int_{-\infty}^\infty\overline{c_\mathfrak{a}(\mu,1/2+it)}  \Gamma (2s-1+2it) \Gamma (2s-1-2it)  E_{\mathfrak{a}} \left(\cdot, 1/2+it\right) dt, \big(|\theta|^2-\Lambda^T|\theta|^2\big)\right\rangle
\\ = \frac{1}{2\pi} \int_{-\infty}^\infty\overline{c_\mathfrak{a}(\mu,1/2+it)}  \Gamma (2s-1+2it) \Gamma (2s-1-2it)
\left\langle E_{\mathfrak{a}} \left(\cdot, 1/2+it\right), |\theta|^2-\Lambda^T|\theta|^2\right\rangle dt.
\end{multline}

We now compute the inner product $\left\langle E_{\mathfrak{a}} \left(\cdot, 1/2+it\right), |\theta|^2-\Lambda^T|\theta|^2\right\rangle$,
and then, by taking the limit $T\to \infty$ in \eqref{e:inner_scrE_withT},
we compute $\left\langle\mathcal{E}_{\mathfrak{a}}(\cdot; \mu), |\theta|^2\right\rangle$.

\begin{lemma}\label{lem:scrEa_theta^2}
For each $\mathfrak{a}\in \mathcal{S}$, for $\mathrm{Re}(s)=\alpha>\frac{2}{3}$,
\begin{multline}\label{lastline}
\left\langle\mathcal{E}_{\mathfrak{a}}(\cdot; \mu, s), |\theta|^2\right\rangle
= \frac{9\sqrt{3}}{2} |c_{\mathfrak{a}}|^2
\frac{\overline{\tilde{c}_\mathfrak{a}}(\mu, 2/3)}{\zeta_K^*(4/3)}
\Gamma \left(2s-1+\frac{1}{3}\right)\Gamma \left(2s-1-\frac{1}{3}\right)
\\ + \frac{1}{2\pi}\int_{-\infty}^\infty
\frac{\overline{\tilde{c}_\mathfrak{a}(\mu,1/2+it)}}{\zeta_K^*(1-2it)}
\frac{3^{9-it}}{2} \frac{(1+3^{-2it})(1-3^{-\frac{1}{2}-it})}{(1-3^{-1-2it})}
\frac{\zeta_K^*(1/2+3it)\zeta_K^*(1/2+it)}{\zeta_K^*(1+2it)}
\\ \times \Gamma (2s-1+2it) \Gamma (2s-1-2it)  dt.
\end{multline}
\end{lemma}

\begin{proof}
By \cite[Lemma~3.1]{Gol},
\begin{equation}
\frac{1}{2\pi i} \int_{(2)} \frac{1}{u^{1+v}}\left(\frac{y}{T}\right)^{u} du
= \begin{cases} \frac{\left(\log \left(\frac{y}{T}\right)\right)^{v}}{\Gamma\left(1+v\right)} & \text{ when } y>T, \\
0 & \text{ otherwise. } \end{cases}
\end{equation}
and the integral converges absolutely for any $\mathrm{Re}(v)>0$.
So
\begin{equation}
1_T(y)
= \lim_{v\to 0} \bigg(\frac{1}{2\pi i} \int_{(2)} \frac{1}{u^{1+v}}\left(\frac{y}{T}\right)^{u} du\bigg).
\end{equation}

For $\mathrm{Re}(v)>0$, let
\begin{equation}
\Lambda^T_{\mathfrak{a}} |\theta|^2(w; v)
= \sum_{\gamma\in \Gamma_{\mathfrak{a}}\backslash \Gamma} \frac{1}{2\pi i} \int_{(2)} \frac{T^{-u}}{u^{1+v}} H(\sigma_{\mathfrak{a}}^{-1} \gamma w)^{\frac{4}{3}+u} du.
\end{equation}
Then $\Lambda_{\mathfrak{a}}^T|\theta|^2(w) =\lim_{v\to 0} \Lambda_{\mathfrak{a}}^T|\theta|^2(w; v)$.
Since the series for $\gamma$ and the integral converge absolutely, we change the order and get
\begin{equation}
\Lambda^T_{\mathfrak{a}} |\theta|^2(w; v)
= \frac{1}{2\pi i} \int_{(2)} \frac{T^{-u}}{u^{1+v}} \sum_{\gamma\in \Gamma_{\mathfrak{a}}\backslash \Gamma}H(\sigma_{\mathfrak{a}}^{-1} \gamma w)^{\frac{4}{3}+u} du
= \frac{1}{2\pi i} \int_{(2)} \frac{T^{-u}}{u^{1+v}} E_{\mathfrak{a}}(w; 2/3+u/2) du.
\end{equation}
By \eqref{exp:eis},
\begin{multline*}
\Lambda^T_{\mathfrak{a}} |\theta|^2(w; v)
= \delta_{\mathfrak{a},\infty}\delta_{y>T} y^{\frac{4}{3}}  \frac{\left(\log \left(\frac{y}{T}\right)\right)^v}{\Gamma\left(1+v\right)}
+ \frac{1}{2\pi i} \int_{(2)} \frac{T^{-u}}{u^{1+v}} c_{\mathfrak{a}}(0, 2/3+u/2) y^{\frac{2}{3}-u} du
\\ +\sum_{0\neq \nu \in \lambda^{-3} \Lambda}
\frac{1}{2\pi i} \int_{(2)} \frac{T^{-u}}{u^{1+v}} c_\mathfrak{a}(\mu, 2/3+u/2) y K_{\frac{1}{3}+u} (4\pi |\nu| y) e(\nu x) du.
\end{multline*}
Note that the integrals converge absolutely as $v\to 0$ (and at $v=0$), so we get the following Fourier expansion for $\Lambda^T_{\mathfrak{a}} |\theta|^2(w)$:
\begin{multline}\label{e:LambdaTtheta^2_Fourier}
\Lambda^T_{\mathfrak{a}} |\theta|^2(w)
= \delta_{\mathfrak{a},\infty}1_T(y) y^{\frac{4}{3}}
+ \frac{1}{2\pi i} \int_{(2)} \frac{T^{-u}}{u} c_{\mathfrak{a}}(0, 2/3+u/2) y^{\frac{2}{3}-u} du
\\ +\sum_{0\neq \nu \in \lambda^{-3} \Lambda}
\frac{1}{2\pi i} \int_{(2)} \frac{T^{-u}}{u} c_\mathfrak{a}(\nu, 2/3+u/2) y K_{\frac{1}{3}+u} (4\pi |\nu| y) du e(\nu x).
\end{multline}

For $z\in \mathbb{C}$, $\mathrm{Re}(z)> \frac{1}{2}$, by unfolding,
\[
\left\langle E_{\mathfrak{a}} \left(\cdot, 1/2+z\right), |\theta|^2-\Lambda^T|\theta|^2\right\rangle
= \int_{\Gamma_{\infty}\backslash \mathbb{H}^3}
\bigg\{|\theta(\sigma_{\mathfrak{a}} w)|^2 - \sum_{\mathfrak{b}\in \mathcal{S}} |c_{\mathfrak{b}}|^2 \Lambda_{\mathfrak{b}}^T|\theta|^2(\sigma_{\mathfrak{a}} w)\bigg\}
y^{1+2z} dV
\]
For any $\mathfrak{a}\in \mathcal{S}$, by \eqref{e:Eab=Ec}, there exists $\mathfrak{c}\in \mathcal{S}$ such that
\begin{align*}
\Lambda^T_{\mathfrak{b}} |\theta|^2(\sigma_{\mathfrak{a}} w; v)
& = \frac{1}{2\pi i} \int_{(2)} \frac{T^{-u}}{u^{1+v}} E_{\mathfrak{b}}(\sigma_{\mathfrak{a}} w; 2/3+u/2) du
= \frac{1}{2\pi i} \int_{(2)} \frac{T^{-u}}{u^{1+v}} E_{\mathfrak{c}}( w; 2/3+u/2) du
\\ & = \Lambda^T_{\mathfrak{c}} |\theta|^2(w; v).
\end{align*}
Note that $\mathfrak{c}=\infty$ when $\mathfrak{a}=\mathfrak{b}$.
By \eqref{e:LambdaTtheta^2_Fourier}, we get
\begin{multline*}
\left\langle E_{\mathfrak{a}} \left(\cdot, 1/2+z\right), |\theta|^2-\Lambda^T|\theta|^2\right\rangle
= {\rm Vol}(3\Lambda)
\int_{0}^\infty
\bigg\{|c_{\mathfrak{a}}|^2 y^{\frac{4}{3}} - |c_{\mathfrak{a}}|^2 1_T(y) y^{\frac{4}{3}}
+ \sum_{0\neq\nu\in \lambda^{-3}\Lambda} |\tau_{\mathfrak{a}}(\nu)|^2 y^2 K_{\frac{1}{3}}(4\pi |\nu|y)^2
\\ - |c_{\mathfrak{a}}|^2\big(\Lambda_{\mathfrak{a}}^T|\theta|^2(\sigma_{\mathfrak{a}} w)- 1_T(y) y^{\frac{4}{3}} \big)
- \sum_{\mathfrak{a}\neq \mathfrak{b}\in \mathcal{S}} |c_{\mathfrak{b}}|^2 \Lambda_{\mathfrak{b}}^T|\theta|^2(\sigma_{\mathfrak{a}} w)
\bigg\}
y^{1+2z} \frac{dy}{y^3}
\\ = {\rm Vol}(3\Lambda) \bigg\{|c_{\mathfrak{a}}|^2 \frac{T^{\frac{1}{3}+2z}}{\frac{1}{3}+2z}
+ (4\pi)^{-1-2z} \sum_{0\neq\nu\in \lambda^{-3}\Lambda} \frac{|\tau_{\mathfrak{a}}(\nu)|^2}{|\nu|^{1+2z}}
\int_0^\infty K_{\frac{1}{3}}(y)^2 y^{1+2z} \frac{dy}{y}\bigg\}
+ O_z(T^{-2}).
\end{multline*}

By Lemma~\ref{lem:theta_fourier}, we have $|\tau_{\mathfrak{a}}(\nu)|^2 = |\tau(\nu)|^2$.
By \eqref{e:L_theta^2} and \eqref{e:K1/3_Mellin}, for ${\rm Vol}(3\Lambda)=\frac{9\sqrt{3}}{2}$, we get
\begin{multline*}
{\rm Vol}(3\Lambda)  (4\pi)^{-1-2z}\sum_{0\neq\nu\in \lambda^{-3}\Lambda} \frac{|\tau_{\mathfrak{a}}(\nu)|^2}{|\nu|^{1+2z}}
\int_0^\infty K_{\frac{1}{3}}(y)^2 y^{1+2z} \frac{dy}{y}
\\ = 3^{9-z} 2^{-1}
\frac{(1+3^{-2z})(1-3^{-\frac{1}{2}-z})}{(1-3^{-1-2z})}
\frac{\zeta_K^*(1/2+3z)\zeta_K^*(1/2+z)}{\zeta_K^*(1+2z)}.
\end{multline*}
Therefore, we get
\begin{multline*}
\left\langle E_{\mathfrak{a}} \left(\cdot, 1/2+z\right), |\theta|^2-\Lambda^T|\theta|^2\right\rangle
= \frac{9\sqrt{3}}{2} |c_{\mathfrak{a}}|^2 \frac{T^{\frac{1}{3}+2z}}{\frac{1}{3}+2z}
\\ + 3^{9-z} 2^{-1}
\frac{(1+3^{-2z})(1-3^{-\frac{1}{2}-z})}{(1-3^{-1-2z})}
\frac{\zeta_K^*(1/2+3z)\zeta_K^*(1/2+z)}{\zeta_K^*(1+2z)}
+ O_z(T^{-2}).
\end{multline*}
The inner product has a meromorphic continuation to $z\in \mathbb{C}$.

Applying to \eqref{e:inner_scrE_withT},
\begin{multline*}
\left\langle\mathcal{E}_{\mathfrak{a}}(\cdot; \mu, s), |\theta|^2-\Lambda^T|\theta|^2\right\rangle
= \frac{9\sqrt{3}}{2}  |c_{\mathfrak{a}}|^2 \frac{1}{2\pi i} \int_{(0)}
\frac{\overline{\tilde{c}_\mathfrak{a}}(\mu,1/2-z)}{\zeta_K^*(1-2z)}
\Gamma (2s-1+2z) \Gamma (2s-1-2z)
\frac{T^{\frac{1}{3}+2z}}{\frac{1}{3}+2z} dz
\\ + \frac{1}{2\pi i} \int_{(0)}
\frac{\overline{\tilde{c}_\mathfrak{a}}(\mu,1/2-z)}{\zeta_K^*(1-2z)}
\frac{3^{9-z}}{2} \frac{(1+3^{-2z})(1-3^{-\frac{1}{2}-z})}{(1-3^{-1-2z})}
\frac{\zeta_K^*(1/2+3z)\zeta_K^*(1/2+z)}{\zeta_K^*(1+2z)}
\\ \times \Gamma (2s-1+2z) \Gamma (2s-1-2z)  dz
+ O(T^{-2}).
\end{multline*}
We move the $z$-line of integration of the first integral to $\mathrm{Re}(z) = -\alpha+\frac{1}{2}+\epsilon$, without passing over poles, except $z=-\frac{1}{6}$:
\begin{multline*}
\frac{1}{2\pi i} \int_{(0)}
\frac{\overline{\tilde{c}_\mathfrak{a}}(\mu,1/2-z)}{\zeta_K^*(1-2z)}
\Gamma (2s-1+2z) \Gamma (2s-1-2z)
\frac{T^{\frac{1}{3}+2z}}{\frac{1}{3}+2z} dz
\\ =
\frac{\overline{\tilde{c}_\mathfrak{a}}(\mu, 2/3)}{\zeta_K^*(4/3)}
\Gamma \left(2s-1+\frac{1}{3}\right)\Gamma \left(2s-1-\frac{1}{3}\right)
\\ + \frac{1}{2\pi i} \int_{(-\alpha+\frac{1}{2}+\epsilon)}
\frac{\overline{\tilde{c}_\mathfrak{a}}(\mu,1/2-z)}{\zeta_K^*(1-2z)}
\Gamma (2s-1+2z) \Gamma (2s-1-2z)
\frac{T^{\frac{1}{3}+2z}}{\frac{1}{3}+2z} dz.
\end{multline*}
Note that in the remaining integral, $\mathrm{Re}(z) = -\alpha+\frac{1}{2}+\epsilon$, so $\frac{1}{3}+2\mathrm{Re}(z) = \frac{4}{3}-2\alpha+2\epsilon<0$ for $\mathrm{Re}(s) = \alpha>\frac{2}{3}$.
So we get
\begin{multline*}
\frac{1}{2\pi i} \int_{(0)}
\frac{\overline{\tilde{c}_\mathfrak{a}}(\mu,1/2-z)}{\zeta_K^*(1-2z)}
\Gamma (2s-1+2z) \Gamma (2s-1-2z)
\frac{T^{\frac{1}{3}+2z}}{\frac{1}{3}+2z} dz
\\ =
\frac{\overline{\tilde{c}_\mathfrak{a}}(\mu, 2/3)}{\zeta_K^*(4/3)}
\Gamma \left(2s-1+\frac{1}{3}\right)\Gamma \left(2s-1-\frac{1}{3}\right)
+ O(T^{\frac{4}{3}-2\alpha+2\epsilon}).
\end{multline*}
Then we get
\begin{multline*}
\left\langle\mathcal{E}_{\mathfrak{a}}(\cdot; \mu, s), |\theta|^2-\Lambda^T|\theta|^2\right\rangle
= \frac{9\sqrt{3}}{2}  |c_{\mathfrak{a}}|^2
\frac{\overline{\tilde{c}_\mathfrak{a}}(\mu, 2/3)}{\zeta_K^*(4/3)}
\Gamma \left(2s-1+\frac{1}{3}\right)\Gamma \left(2s-1-\frac{1}{3}\right)
\\ + \frac{1}{2\pi i} \int_{(0)}
\frac{\overline{\tilde{c}_\mathfrak{a}}(\mu,1/2-z)}{\zeta_K^*(1-2z)}
\frac{3^{9-z}}{2} \frac{(1+3^{-2z})(1-3^{-\frac{1}{2}-z})}{(1-3^{-1-2z})}
\frac{\zeta_K^*(1/2+3z)\zeta_K^*(1/2+z)}{\zeta_K^*(1+2z)}
\\ \times \Gamma (2s-1+2z) \Gamma (2s-1-2z)  dz
+ O(T^{\frac{4}{3}-2\alpha+2\epsilon}).
\end{multline*}
By taking $T\to \infty$, we get
\begin{multline*}
\lim_{T\to \infty} \left\langle\mathcal{E}_{\mathfrak{a}}(\cdot; \mu, s), |\theta|^2-\Lambda^T|\theta|^2\right\rangle
= \frac{9\sqrt{3}}{2} |c_{\mathfrak{a}}|^2
\frac{\overline{\tilde{c}_\mathfrak{a}}(\mu, 2/3)}{\zeta_K^*(4/3)}
\Gamma \left(2s-1+\frac{1}{3}\right)\Gamma \left(2s-1-\frac{1}{3}\right)
\\ + \frac{1}{2\pi}\int_{-\infty}^\infty
\frac{\overline{\tilde{c}_\mathfrak{a}(\mu,1/2+it)}}{\zeta_K^*(1-2it)}
\frac{3^{9-it}}{2} \frac{(1+3^{-2it})(1-3^{-\frac{1}{2}-it})}{(1-3^{-1-2it})}
\frac{\zeta_K^*(1/2+3it)\zeta_K^*(1/2+it)}{\zeta_K^*(1+2it)}
\\ \times \Gamma (2s-1+2it) \Gamma (2s-1-2it)  dt.
\end{multline*}
\end{proof}

We now need to estimate the right-hand side of \eqref{lastline} and complete the proof of Lemma~\ref{lemma:cont}.
Recall \eqref{lastline} and name the integrals
\begin{multline}
\left\langle\mathcal{E}_{\mathfrak{a}}(\cdot; \mu, s), |\theta|^2\right\rangle
= \frac{9\sqrt{3}}{2} |c_{\mathfrak{a}}|^2
\frac{\overline{\tilde{c}_\mathfrak{a}}(\mu, 2/3)}{\zeta_K^*(4/3)}
\Gamma \left(2s-1+\frac{1}{3}\right)\Gamma \left(2s-1-\frac{1}{3}\right)
\\ + \frac{1}{2\pi}\int_{-\infty}^\infty
\frac{\overline{\tilde{c}_\mathfrak{a}(\mu,1/2+it)}}{\zeta_K^*(1-2it)}
\frac{3^{9-it}}{2} \frac{(1+3^{-2it})(1-3^{-\frac{1}{2}-it})}{(1-3^{-1-2it})}
\frac{\zeta_K^*(1/2+3it)\zeta_K^*(1/2+it)}{\zeta_K^*(1+2it)}
\\ \times \Gamma (2s-1+2it) \Gamma (2s-1-2it)  dt
=: I+ II.
\end{multline}

The first piece, coming from the residue is easy.
As $s= \alpha +ir$, by Stirling's formula
\[
I= \frac{9\sqrt{3}}{2} |c_{\mathfrak{a}}|^2
\frac{\overline{\tilde{c}_\mathfrak{a}}(\mu, 2/3)}{\zeta_K^*(4/3)}
\Gamma \left(2s-1+\frac{1}{3}\right)\Gamma \left(2s-1-\frac{1}{3}\right) \ll_\mu e^{-2\pi |r|}(1+|r|)^{4\alpha -3}.
\]
To estimate $II$ we first apply the lower bound
\[
\zeta_K(1\pm 2it) \gg \left(\log(2+2|t|)\right)^{-2}.
\]
This follows as $\zeta_K(1\pm 2it) = \zeta(1\pm 2it)L(1\pm 2it, \chi_{-3})$ and
$$
\zeta(1\pm 2it),L(1\pm 2it, \chi_{-3}) \gg  \left(\log(1+2|t|)\right)^{-1}.
$$
We then apply Stirling (Lemma~\ref{Stirling}) and obtain
\[
II \ll_\mu \int_{-\infty}^\infty (1+|t|)^{-1+\epsilon}( |r+t| +1)^{2\alpha-\frac32}( |r-t| +1)^{2\alpha-\frac32}e^{-2\pi\max(|r|,|t|)}\left|\zeta_K(1/2+3it)\zeta_K(1/2 +it)\right|dt.
\]
Here we have absorbed the $|\log(1+2|t|)|^2$ in the $\epsilon$ of $(1+|t|)^{-1+\epsilon}$.
Because of the exponential decay when $|t| >|r|$ and the polynomial growth of the rest of the expression in $t$, we have
\begin{equation}\label{II}
II \ll_\mu (1+ |r|)^{4\alpha-3} e^{-2\pi|r|} \int_{-|r|}^{|r|} (1+|t|)^{-1+\epsilon}\left|\zeta_K(1/2+3it)\zeta_K(1/2 +it)\right|dt.
\end{equation}
We estimate this by integration by parts, using the integral theorem for Dirichlet polynomials \cite[Theorem 9.1]{IK04},
\[
\int_0^T \left|\sum_{1\le n\le N} a_nn^{it}\right|^2dt \ll \left( T+ \mathcal{O}(N)\right)\sum_{1\le n\le N} |a_n|^2.
\]
The conductors of $\zeta_K(1/2+3it),\zeta_K(1/2 +it)$ are both $|t|$, and so each can be represented as a sum of length a multiple of $|r|$, as $|t| \le |r|$.  It follows then from Cauchy-Schwartz and the above that
$$
 \int_{-|r|}^{|r|} \left|\zeta_K(1/2+3it)\zeta_K(1/2 +it)\right|dt \ll \left( \int_{-|r|}^{|r|} \left|\zeta_K(1/2+3it)\right|^2dt\right)^{\frac12}\left( \int_{-|r|}^{|r|} \left|\zeta_K(1/2+it)\right|^2dt\right)^{\frac12}.
 $$
Applying the approximate functional equation (the pole of the zeta function does not affect the estimate), and the integral theorem we have
$$
 \int_{-|r|}^{|r|} \left|\zeta_K(1/2+it)\right|^2dt \ll\int_{-|r|}^{|r|}  \left|\sum_{1\le n\le |r|} a_nn^{it}\right|^2dt \sum_{1\le n\le |r|} \frac{1}{n}\ll|r|^{1+\epsilon}.
$$
Here the $a_n$ are the coefficients of $\zeta_K$.
The same estimate applies to $\int_{-|r|}^{|r|} \left|\zeta_K(1/2+3it)\right|^2dt$, and so
\begin{equation}\label{intbound}
 \int_{-|r|}^{|r|} \left|\zeta_K(1/2+3it)\zeta_K(1/2 +it)\right|dt \ll |r|^{1+\epsilon}
\end{equation}
We now integrate
\[
\int_{-|r|}^{|r|} (1+|t|)^{-1+\epsilon}\left|\zeta_K(1/2+3it)\zeta_K(1/2 +it)\right|dt \ll
\int_{1}^{|r|} t^{-1+\epsilon}\left|\zeta_K(1/2+3it)\zeta_K(1/2 +it)\right|dt
\]
by parts, setting
\[
S(t) = \int_{1}^{t }\left|\zeta_K(1/2+3it')\zeta_K(1/2 +it')\right|dt'.
\]
Then
\begin{align*}
\int_{1}^{|r|} t^{-1+\epsilon}\left|\zeta_K(1/2+3it)\zeta_K(1/2 +it)\right|dt
& = \int_{1}^{|r|} t^{-1+\epsilon}dS(t)
\\ & = \left[ t^{-1+\epsilon}S(t)\right]_1^{|r|}  +(1-\epsilon) \int_{1}^{|r|} t^{-2+\epsilon}S(t)dt
\\ & \ll |r|^{2\epsilon},
\end{align*}
after applying \eqref{intbound}.

Combining this with \eqref{II} finally gives us
\[
II \ll_{\mu}(1+ |r|)^{4\alpha-3 + \epsilon} e^{-2\pi|r|},
\]
and completes the proof.

\section{Completion of the proof of Theorem \ref{theorem:main}}
We are going to compute the inner product $\langle  P_\mu(\cdot, s),|\theta|^2 \rangle$ directly, and then compare it with the summation formula from Lemma \ref{lemma:spec}. This will complete the proof of Theorem \ref{theorem:main}.
\begin{lemma}\label{lemma:direct}
Let $s=\alpha+ir$ with $\alpha$ being large and fixed. Assume that $\mu$ is chosen and fixed such that $\mathrm{Re}(\tau(\mu))\neq 0$. Then we have
\[
|\langle P_\mu(\cdot, s) ,|\theta|^2  \rangle| \sim_{\alpha,\mu} (1+|r|)^{2\alpha - \frac{4}{3}} e^{-\pi |r|},
\]
as $r\to \infty$.
\end{lemma}
\begin{proof}
We first unfold the integral and then represent the inner product as a summation of shifted convolution sums as follows:
\begin{align*}
\langle P_{\mu}(\cdot,s),|\theta|^2 \rangle &= \iiint_{\Gamma \backslash \mathbb{H}^3} P_{\mu}(w,s) |\theta(w)|^2\frac{dx_1dx_2dy}{y^3}\\
&=\iiint_{\Gamma_\infty  \backslash \mathbb{H}^3} (8\pi |\mu| y)^{2s}e^{-4\pi|\mu| y} e(\mu x)|\theta(w)|^2 \frac{dx_1dx_2dy}{y^3}\\
&= \sigma \left(\tau(-\mu)+\overline{\tau(\mu)}\right)\int_0^\infty   K_{\frac{1}{3}} (4\pi |\mu| y) (8\pi |\mu| y)^{2s}e^{-4\pi|\mu| y} y^{-\frac{1}{3}} \frac{dy}{y}  \\
&+\sum_{\substack{ \nu \in \lambda^{-3} \Lambda\\ \nu \neq 0,-\mu }} \tau(\nu)\overline{\tau(\nu+\mu)} \int_0^\infty K_{\frac{1}{3}} (4\pi |\nu| y)   K_{\frac{1}{3}} (4\pi |\nu+\mu| y)(8\pi |\mu| y)^{2s}e^{-4\pi|\mu| y} \frac{dy}{y}\\
&= \sigma \left(\tau(-\mu)+\overline{\tau(\mu)}\right)(4\pi |\mu|)^{\frac{1}{3}} \int_0^\infty  K_{\frac{1}{3}} ( y) (2 y)^{2s}e^{- y} y^{-\frac{4}{3}} dv  \\
&+\sum_{\substack{ \nu \in \lambda^{-3} \Lambda\\ \nu \neq 0,-\mu }} \tau(\nu)\overline{\tau(\nu+\mu)} \int_0^\infty K_{\frac{1}{3}} \left( \frac{|\nu|}{|\mu|} y\right)   K_{\frac{1}{3}} \left( \frac{|\nu+\mu|}{|\mu|} y\right)(2  y)^{2s}e^{- y} \frac{dy}{y}\\
&=I+II.
\end{align*}
For the first integral, we use \eqref{eq1} so that
\[
I = 2 \pi^{\frac{5}{6}} \sigma \left(\tau(-\mu)+\overline{\tau(\mu)}\right) |\mu|^{\frac{1}{3}}
\frac{\Gamma (2s) \Gamma \left(2s-\frac{2}{3}\right)}{\Gamma\left(2s+\frac{1}{6}\right)}.
\]
If we take $s = \alpha + ir$, then from Stirling's approximation (Lemma \ref{Stirling}),
\[
|I| \sim \sigma |\mu|^{\frac{1}{3}}  |\tau(-\mu)+\overline{\tau(\mu)}| (1+|r|)^{2\alpha-\frac{4}{3}} e^{-\pi |r|}.
\]
In Appendix~\ref{appendix}, we give proof of a crude estimate
\begin{equation}\label{stationary}
 \int_0^\infty K_{\frac{1}{3}} \left( \frac{|\nu|}{|\mu|} y\right)   K_{\frac{1}{3}} \left( \frac{|\nu+\mu|}{|\mu|} y\right)(2  y)^{2s}e^{- y} \frac{dy}{y}\ll_{\mu,\alpha} \frac{e^{-\pi |r|} (1+|r|)^{2\alpha -\frac{3}{2}} \log (1+|r|)}{(|\mu|+|\nu|+|\mu+\nu|)^{2\alpha-1}}.
\end{equation}
When combined with a trivial estimate $\tau(\mu) \ll |\mu|^{\frac{1}{3}}$, it implies that
\[
II\ll e^{-\pi |r|} (1+|r|)^{2\alpha-\frac{3}{2}} \log (2+|r|)
\]
provided that $\alpha$ is sufficiently large (say, $\alpha>10$).
So the statement follows from the observation that $\tau(\mu)=\tau(-\mu)$.
\end{proof}

We now prove Theorem \ref{theorem:main}. We first  fix $\mu \neq 0$ such that $\mathrm{Re}(\tau(\mu))\neq 0$ and a large $\alpha>10$. One can take for instance $\mu=1$ and $\alpha=100$. For such $\mu$ and $\alpha$, by Lemma \ref{lemma:spec} and Lemma \ref{lemma:direct},
\begin{multline*}
(1+|r|)^{2\alpha - \frac{4}{3}} e^{-\pi |r|} \ll_{\alpha,\mu} \sum_{j\geq 1 }   \frac{\Gamma (2s-1+2it_j) \Gamma (2s-1-2it_j)}{\Gamma \left(2s - \frac{1}{2}\right)}  \overline{\rho_j(\mu)} \langle \phi_j,|\theta|^2 \rangle \\
\ll_{\epsilon,\mu} \sum_{\substack{j\geq 1\\  \langle \phi_j,|\theta|^2 \rangle  \neq 0 }} e^{-\pi(|r+t_j|+|r-t_j|-|r|+|t_j|)} ((1+|r+t_j|)(1+|r-t_j|))^{2\alpha -\frac{3}{2}} (1+|r|)^{-2\alpha+1} (1+|t_j|)^{\frac{4}{3}+\epsilon},
\end{multline*}
where we used \eqref{triv} and Lemma \ref{firstcoe} in the second estimate.
Assume for contradiction that there are only finitely many $j$'s such that $\langle \phi_j,|\theta|^2 \rangle  \neq 0$. Then the right-hand side is
\[
\ll e^{-\pi |r|} (1+|r|)^{2\alpha-2}
\]
as $r \to \infty$, which cannot happen because $2\alpha - \frac{4}{3 }  > 2\alpha - 2 $.
This completes the proof of Theorem \ref{theorem:main} using Theorem \ref{theorem:mainlem}.

\appendix
\section{Proof of \texorpdfstring{\eqref{stationary}}{(3.1)}}\label{appendix}
Here we give a crude estimate of
\[
 \int_0^\infty K_{\frac{1}{3}} \left( \frac{|\nu|}{|\mu|} y\right)  K_{\frac{1}{3}} \left( \frac{|\nu+\mu|}{|\mu|} y\right) (2  y)^{2s}e^{- y} \frac{dy}{y},
\]
when $\nu \neq 0,-\mu$, which is used in Lemma \ref{lemma:direct}. When $\mathrm{Re}(s)$ is fixed, it is possible to obtain an asymptotic expansion uniform in $\mu,\nu,\mathrm{Im}(s)$ using a standard technique from harmonic analysis (see for instance \cite[Ch. VII \S2]{MR1232192}), hence it is possible to obtain a sharper estimate than the estimate we prove here. However, the proof of the weaker estimate \eqref{stationary} is much simpler and sufficient for our application.

To begin with, we recall that
\[
K_{\frac{1}{3}} (x) = \sqrt{3} \int_0^\infty \exp \left(-x (1+4\xi^2/3)\sqrt{1+\xi^2/3}\right) d\xi.
\]
Let $f(\xi) = (1+4\xi^2/3)\sqrt{1+\xi^2/3}$ and substitute $\frac{|\nu|}{|\mu|}$ and $\frac{|\nu+\mu|}{|\mu|}$ by $a$ and $b$ respectively. We then express the integral as
\begin{align*}
\int_0^\infty K_{\frac{1}{3}} ( a y)  & K_{\frac{1}{3}} ( b y)(2  y)^{2s} e^{- y} \frac{dy}{y} = 3 \int_0^\infty \int_0^\infty \int_0^\infty e^{-y \left(af(\xi_1)+bf(\xi_2)\right)}(2  y)^{2s}e^{- y} \frac{dy}{y}d\xi_1 d\xi_2\\
&=3\cdot 4^s \int_0^\infty \int_0^\infty \int_0^\infty\left( 1+ af(\xi_1)+bf(\xi_2)\right)^{-2s}y^{2s}e^{- y} \frac{dy}{y}d\xi_1 d\xi_2\\
&=3\cdot 4^{s-1}\Gamma(2s) \iint_{\mathbb{R}^2}\left( 1+ af(\xi_1)+bf(\xi_2)\right)^{-2s}d\xi_1 d\xi_2.
\end{align*}
Now let $s=\alpha +ir$ with $\alpha>100$ being fixed, and let
\[
 \iint_{\mathbb{R}^2}\left( 1+ af(\xi_1)+bf(\xi_2)\right)^{-2s}d\xi_1 d\xi_2 =  \iint_{\mathbb{R}^2}g(\xi_1,\xi_2) e^{-2ir\phi(\xi_1,\xi_2)}d\xi_1 d\xi_2
\]
where
\[
g(\xi_1,\xi_2) = \left( 1+ af(\xi_1)+bf(\xi_2)\right)^{-2\alpha}
\]
and
\[
\phi(\xi_1,\xi_2)  = \log \left( 1+ af(\xi_1)+bf(\xi_2)\right).
\]
Let $\psi \in C_0^\infty (\mathbb{R})$ be a nonnegative function such that $\psi(x) = 1$ if $x<1$, $\psi(\xi) = 0$ if $|\xi|>3/2$, and $|\psi'|, |\psi''| < 10$.

Note that $ \phi_{\xi_i} = 0$ if and only if $\xi_i=0$, and so we treat the part containing a stationary phase
\[
M=\iint_{\mathbb{R}^2}\psi(\xi_1)\psi(\xi_2)g(\xi_1,\xi_2) e^{-2ir\phi(\xi_1,\xi_2)}d\xi_1 d\xi_2,
\]
and the rest
\[
R=\iint_{\mathbb{R}^2}(1-\psi(\xi_1)\psi(\xi_2))g(\xi_1,\xi_2) e^{-2ir\phi(\xi_1,\xi_2)}d\xi_1 d\xi_2
\]
separately. We further split $R$ into two integrals:
\begin{align*}
R &= \iint_{\mathbb{R}^2} (1-\psi(\xi_1))\psi(\xi_2)g (\xi_1,\xi_2) e^{-2ir\phi(\xi_1,\xi_2)}d\xi_1 d\xi_2 + \iint_{\mathbb{R}^2}(1-\psi(\xi_2))g(\xi_1,\xi_2) e^{-2ir\phi(\xi_1,\xi_2)}d\xi_1 d\xi_2\\
 &= R_1+R_2.
\end{align*}
For $R_1$, observe that
\[
R_1 = \frac{1}{2ir}\iint_{\mathbb{R}^2} \frac{\partial}{\partial \xi_1} \left( \frac{(1-\psi(\xi_1))\psi(\xi_2)g(\xi_1,\xi_2)}{\phi_{\xi_1}(\xi_1,\xi_2)}\right) e^{-2ir\phi(\xi_1,\xi_2)}d\xi_1 d\xi_2,
\]
where
\begin{align*}
 \frac{\partial}{\partial \xi_1} \left( \frac{(1-\psi(\xi_1))\psi(\xi_2)g(\xi_1,\xi_2)}{\phi_{\xi_1}(\xi_1,\xi_2)}\right)
 &= \frac{\psi(\xi_2)}{a} \frac{\partial}{\partial \xi_1} \left( \frac{(1-\psi(\xi_1))\left( 1+ af(\xi_1)+bf(\xi_2)\right)^{-2\alpha+1}}{f'(\xi_1)}\right)\\
 &= \frac{\psi(\xi_2)}{a} \frac{\partial}{\partial \xi_1} \left( \frac{\sqrt{9+3\xi_1^2}(1-\psi(\xi_1))\left( 1+ af(\xi_1)+bf(\xi_2)\right)^{-2\alpha+1}}{9\xi_1+4\xi_1^3}\right)\\
 &\ll \alpha \left(b+ 1+ af(\xi_1)\right)^{-2\alpha},
\end{align*}
which holds uniformly in $|\xi_2| <3/2$. Therefore
\[
|R_1| \ll  \frac{\alpha}{r} \int_1^\infty  \left( b+1+ af(\xi_1)\right)^{-2\alpha} d\xi_1 \ll_\mu \frac{1}{r}(a+b+1)^{-2\alpha}.
\]
Likewise, we perform integration by parts with respect to $\xi_2$ to infer that
\[
|R_2| \ll_\mu \frac{1}{r}(a+b+1)^{-2\alpha+1}.
\]
Now for the main contribution $M$, we first integrate by parts with respect to $\xi_1$ and then $\xi_2$ to see that
\[
M \ll \sup_{|r_1|,|r_2|<3/2}\int_{-2}^{r_1}\int_{-2}^{r_2} e^{-2ir\phi(\xi_1,\xi_2)}  d\xi_1 d\xi_2 \int_{-2}^2\int_{-2}^2  \left|\frac{\partial^2}{\partial \xi_1\partial \xi_2} \left(\psi(\xi_1)\psi(\xi_2)g(\xi_1,\xi_2)\right)\right| d\xi_1 d\xi_2.
\]
Note that
\[
|\phi_{\xi_i\xi_i}| \sim_\mu r,~ |\phi_{\xi_1\xi_2}| \ll_\mu r,~ \left|\frac{\partial(\phi_{\xi_1},\phi_{\xi_2})}{\partial(\xi_1,\xi_2)}\right| \gg r^2
\]
as $r\to \infty$, so we apply the Lemma $\delta$ of Titchmarsh \cite{tit34} to obtain the following estimate
\[
\sup_{|r_1|,|r_2|<3/2}\int_{-2}^{r_1}\int_{-2}^{r_2} e^{-2ir\phi(\xi_1,\xi_2)}  d\xi_1 d\xi_2  \ll_\mu \frac{\log r}{r}
\]
as $r \to \infty$. Combining with the following estimate
\begin{align*}
&\int_{-2}^2\int_{-2}^2  \left|\frac{\partial^2}{\partial \xi_1\partial \xi_2} \left(\psi(\xi_1)\psi(\xi_2)g(\xi_1,\xi_2)\right)\right| d\xi_1 d\xi_2\\
\ll &\int_{-2}^2\int_{-2}^2  \left|g(\xi_1,\xi_2)\right|+\left|g_{\xi_1}(\xi_1,\xi_2)\right|+\left|g_{\xi_2}(\xi_1,\xi_2)\right|+\left|g_{\xi_1\xi_2}(\xi_1,\xi_2)\right| d\xi_1 d\xi_2\\
\ll &\alpha^2 (1+a+b)^{-2\alpha},
\end{align*}
we see that
\[
M \ll_{\mu,\alpha} \frac{\log r}{r}(1+a+b)^{-2\alpha}.
\]
This proves the estimate
\[
 \int_0^\infty K_{\frac{1}{3}}\left(\frac{|\nu|}{|\mu|} y\right) K_{\frac{1}{3}}\left(\frac{|\nu+\mu|}{|\mu|} y\right)(2  y)^{2s}e^{- y} \frac{dy}{y} \ll_{\mu,\alpha} \frac{e^{-\pi |r|} (1+|r|)^{2\alpha -\frac{3}{2}} \log (1+|r|)}{(|\mu|+|\nu|+|\mu+\nu|)^{2\alpha-1}},
\]
where we used Stirling's approximation (Lemma \ref{Stirling})
\[
\Gamma(2s)\sim e^{-\pi |r|} (1+|r|)^{2\alpha -\frac{1}{2}},
\]
for $s=\alpha +ir$.

\bibliography{bibfile}
\bibliographystyle{alpha}

\end{document}